\documentclass [12pt]{amsart}
\usepackage[utf8]{inputenc}
\usepackage{amsmath}
\usepackage{amssymb}
\usepackage{amsthm}
\usepackage{graphicx}
\usepackage{amsfonts, amssymb}
\usepackage{comment}
\usepackage{enumitem}                                          
\usepackage{epstopdf}
\usepackage[utf8]{inputenc} 
\usepackage[english]{babel}
\usepackage{subfig}
\usepackage{tikz}
\usetikzlibrary{shapes,snakes,positioning}
\newtheorem{thm}{Theorem}[section]
\newtheorem{lemma}[thm]{Lemma}
\newtheorem{prop}[thm]{Proposition}

\newtheorem{problem}[thm]{Problem}

\theoremstyle{definition}
\newtheorem{remark}[thm]{Remark}
\theoremstyle{definition}
\newtheorem{dfn}[thm]{Definition}
\theoremstyle{definition}
\newtheorem{example}[thm]{Example}

\def\C{\mathbb{C}}
\def\R{\mathbb{R}}

\def\Q{\mathbb{Q}}

\def\II{{\overline{I}}}

\def\P{{\mathbb P}}
\def\F{{\mathbb F}}

\def\<{{\langle}}
\def\>{{\rangle}}

\def\AA{{\overline{A}}}

\begin{document}

\def\ch{{[2k]\choose k}}
\def\rk{{\mathrm{rk}}}
\def\A{{\mathcal{A}}}
\def\S{{\mathcal{S}}}
\newcommand{\LR}[1]{\mathcal{LR}([0,#1])}
\def\V{{\mathcal{V}}}
\def\W{{\mathcal{W}}}
\def\WW{{\widetilde{\mathcal{W}}}}

\def\WO{{\mathcal{W}^\circ}}
\def\Wab{{\WO_{a,b}}}

\def\U{{\mathcal{U}}}
\def\UO{{\mathcal{U}^\circ}}
\def\UW{{\UO(\WO)}}

\def\D{{\mathcal{D}}}
\def\DO{{\mathcal{D}^\circ}}
\def\DW{{\DO(\WO)}}

\def\L{{\mathcal{L}}}
\def\X{{\mathcal{X}}}
\def\Y{{\mathcal{Y}}}
\def\LO{{\mathcal{L}^\circ}}
\def\R{{\mathcal{R}}}
\def\RO{{\mathcal{R}^\circ}}

\def\JR{{\overrightarrow{J}}}
\def\JL{{\overleftarrow{J}}}
\def\I{{\mathcal{I}}}

\def\P{{\mathcal{P}}}
\def\D{{\mathcal{D}}}
\def\F{{\mathcal{F}}}
\def\M{{\mathcal{M}}}
\def\CC{{\mathcal{C}}}
\def\C{{\mathcal{C}}}
\def\B{{\mathcal{B}}}
\def\inn{{\mathrm{in}}}
\def\outt{{\mathrm{out}}}
\def\sp{{\sigma^{-1}\pi}}
\def\QQ{{\widetilde{Q}}}

\def\pii{\pi^{:}}
\def\col{{\mathrm{col}}}

\def\Ired{{\I^\circ}}
\def\act{{\mathrm{N}}}
\def\Sfrak{{\mathfrak{S}}}
\def\Alignments{\mathrm{al}}

\def\It{{\widetilde{I}}}
\def\Jt{{\widetilde{J}}}
\def\Pt{{{P}}}
\def\Qt{{{Q}}}
\def\Rt{{\widetilde{R}}}
\def\St{{\widetilde{S}}}

\title[Weak Separation, Pure Domains and Cluster Distance]{Weak Separation, Pure Domains and Cluster Distance}

\author{Miriam Farber}
\author{Pavel Galashin}
\email{mfarber@mit.edu\\
galashin@mit.edu}
\thanks{The work of the first author was supported by the National Science Foundation Graduate Research Fellowship under Grant No. 1122374}

\address{Department of Mathematics, Massachusetts Institute of Technology, Cambridge MA, USA.}

\date\today
\maketitle

\begin{abstract}
Following the proof of the purity conjecture for weakly separated collections, recent years have revealed a variety of wider examples of purity in different settings. In this paper we consider the collection $\mathcal A_{I,J}$ of sets that are weakly separated from two fixed sets $I$ and $J$. We show that all maximal by inclusion weakly separated collections $\mathcal W\subset\mathcal A_{I,J}$ are also maximal by size, provided that $I$ and $J$ are sufficiently ``generic''. We also give a simple formula for the cardinality of $\mathcal W$ in terms of $I$ and $J$. We apply our result to calculate the cluster distance and to give lower bounds on the mutation distance between cluster variables in the cluster algebra structure on the coordinate ring of the Grassmannian. Using a linear projection that relates weak separation to the octahedron recurrence, we also find the exact mutation distances and cluster distances for a family of cluster variables.
\end{abstract}

\section{Introduction}

In 1998, Leclerc and Zelevinsky introduced the notion of weakly separated collections while studying quasicommuting families of quantum minors (see \cite{LZ}). They raised the ``purity conjecture'', which states that all maximal by inclusion collections of pairwise weakly separated subsets of $[n]:=\{1,2\ldots,n\}$ have the same cardinality. This conjecture was proven independently in \cite{P} and \cite{DKK10}. Since then, it motivated the search for other instances of the purity phenomenon. Such instances have been found in \cite{DKK14} using a novel geometric-combinatorial model called \emph{combined tilings}. Furthermore, the work of \cite{P} showed that all maximal weakly separated collections can be obtained from each other by a sequence of mutations. It was also shown by Scott~\cite{Scott,Scott2} that collections of Pl\"ucker coordinates labeled by maximal weakly separated collections form clusters in the cluster algebra structure on the coordinate ring of the Grassmannian.

In this paper, we study a new instance of the purity phenomenon for the collection $\A_{I,J}$ associated with pairs $I,J$ of subsets of $[n]$. Such a collection arises naturally when one studies a notion of cluster distance between two cluster variables, and specifically variables that are  Pl\"ucker coordinates in the cluster algebra structure on the coordinate ring of the Grassmannian. This purity result allows us to compute this distance for any ``generic'' pair of cluster variables. We further reformulate this distance in the context of general cluster algebras.

Let us first motivate the need for a notion of distance in a cluster algebra, concentrating on the example of the Grassmannian. A pair of Pl\"ucker coordinates can appear together in the same cluster if and only if they are labeled by weakly separated subsets of $[n]$. For Pl\"ucker coordinates that \emph{do not} appear together in the same cluster, we would like to estimate how close they are to being weakly separated. In a more general sense, it would be beneficial to have a notion that would measure how close are two cluster variables from appearing in the same cluster. When they do appear in the same cluster we say that the distance between them is zero.

Section~\ref{sec:Preliminaries} develops this notion (defined in equation (\ref{eqd1})) and introduces the notions of being weakly separated and ``generic'' mentioned above. Section~\ref{sec:Mainresult} states our main results. Section~\ref{section:further} provides the necessary background on domains inside and outside simple closed curves and plabic graphs, used in the proof of our main result. We prove Theorem~\ref{thm:LR} in Section~\ref{section:proofLR}. In Sections~\ref{sect:iso}, \ref{sec:description} and \ref{sec:Chord_separation} we develop some tools that will be useful to us in Section~\ref{section:proof} where we prove Theorems~\ref{thm:main} and \ref{thm:main_nc}. Finally, Section~\ref{section:unbalanced} gives a formula for the mutation distance (introduced in \cite{FP}) for a family of pairs of cluster variables and relates the corresponding optimal sequence of mutations with that of the octahedron recurrence.

\section{Preliminaries}
\label{sec:Preliminaries}
For $0 \leq k \leq n$, we denote by $[n] \choose k$ the collection\footnote{throughout the paper, we reserve the word ``set'' for subsets of $[n]$ while we use the word ``collection'' for subsets of $2^{[n]}$.} of $k$-element subsets of $[n]$. For $I\subset [n]$, we denote by $\bar I=[n]\setminus I$ its complement in $[n]$. For two subsets $I,J\subset [n]$ we write $I<J$ if $\max(I)<\min(J)$. We say that $I$ \emph{surrounds} $J$ if $I \setminus J$ can be partitioned as $I_1 \sqcup I_2$ such that $I_1 < J \setminus I < I_2$. We denote by $I \bigtriangleup J$ the symmetric difference $(I \setminus J) \cup (J \setminus I)$.
\begin{dfn}[\cite{LZ}]
 Two sets $S,T\subset [n]$ are called \emph{weakly separated} if at least one of the following holds:
 \begin{itemize}
  \item $|S|\leq |T|$ and $S$ surrounds $T$;
  \item $|T|\leq |S|$ and $T$ surrounds $S$.
 \end{itemize}
\end{dfn}

\begin{remark}\label{rmk:ws_definition}
 This definition has a particularly simple meaning when $I$ and $J$ have the same size. Consider a convex $n$-gon with vertices labeled by numbers $1$ through $n$. Then it is easy to see that two subsets $I$ and $J$ of the same size are weakly separated iff the convex hull of the vertices from the set $I\setminus J$ does not intersect the convex hull of the vertices from the set $J\setminus I$.
\end{remark}

\begin{dfn}
A collection $\CC \subset 2^{[n]}$ is called \emph{weakly separated} if any two of its elements are weakly separated. It is called \emph{maximal weakly separated} if it is weakly separated and is not contained in any other weakly separated collection.
\end{dfn}

\begin{dfn}[see \cite{DKK14}]
A collection $\A\subset 2^{[n]}$ is called \emph{a pure domain} if all maximal (by inclusion) weakly separated collections of sets from $\A$ have the same size. In this case, the size of all such collections is called \emph{the rank of $\A$} and denoted $\rk \A$.
\end{dfn}

The following two surprising results go under the name ``purity phenomenon'' and were conjectured in \cite{LZ} and \cite{Scott}, respectively. Both of them were proven independently in \cite{P} and in \cite{DKK10}:

\begin{thm}\label{thm:purity_of_n}
The collection $2^{[n]}$ is a pure domain of rank ${n \choose 2}+n+1$.
\end{thm}

\begin{thm}\label{thm:purity_of_nchoosek}
The collection $[n] \choose k$ is a pure domain of rank $k(n-k)+1$.
\end{thm}

The latter result has a stronger version that relates pairs of maximal weakly separated collections by an operation that is called \emph{a mutation}.
\begin{prop}[see \cite{LZ}, \cite{Scott}]\label{prop:mutations}
Let $S \in {[n] \choose k-2}$ and let $a,b,c,d$ be cyclically ordered elements of $[n]\setminus S$. Suppose a maximal weakly separated collection $\mathcal{C}$ contains $Sab, Sbc, Scd, Sda$ and $Sac$. Then $\mathcal{C}'=(\mathcal{C} \setminus \{Sac\}) \cup \{Sbd\}$ is also a maximal weakly separated collection.
\end{prop}

Here by $Sab$ we mean $S\cup\{a,b\}$. We say that $\CC'$ and $\CC$ as above are \emph{connected by a mutation}\footnote{Such a mutation is usually called a \emph{square move}. There is a more general notion of a \emph{cluster mutation}, but it is beyond the scope of this paper.}.
\begin{thm}[see \cite{P}]\label{thm:purity_of_nchoosek_mutations}
The collection $[n] \choose k$ is a pure domain of rank $k(n-k)+1$. Moreover, any two maximal weakly separated collections in $[n] \choose k$ are connected by a sequence of mutations.
\end{thm}

\begin{remark}\label{rmk:simple_closed_curve}
The list of collections known to be pure domains is not restricted to just $2^{[n]}$ and $[n]\choose k$. In \cite{P} both Theorem~\ref{thm:purity_of_n} and Theorem~\ref{thm:purity_of_nchoosek} above can be seen as special cases of the purity phenomenon for the collections inside a \emph{positroid}, while in \cite{DKK10} these collections are inside and outside of a \emph{generalized cyclic pattern}. Note that in both \cite{P} and \cite{DKK10}, the collections are required to lie, in a sense, \emph{inside and outside a specific simple closed curve} (see Definition~\ref{dfn:in_out} and Remark~\ref{rmk:positroid}). In this paper we present a new instance of the purity phenomenon which does not fit into this context, see Figure~\ref{fig:max_collection}. The domains that we discuss arise naturally when dealing with distances between cluster variables in the cluster algebra structure on the coordinate ring of the Grassmannian.
\end{remark}

\begin{dfn}
For $n\geq k\geq 0$, the {\it Grassmannian} $ \operatorname{Gr}(k,n)$ (over $\mathbb{R}$) is the space of $k$-dimensional linear subspaces in $\mathbb{R}^n$. It can be identified with the space of real $k\times n$ matrices of rank $k$ modulo row operations.
The $k\times k$ minors of $k\times n$-matrices form projective coordinates on the Grassmannian, called
the {\it Pl\"ucker coordinates}, that are denoted by $\Delta_I$, where $I \in {[n] \choose k}$.
\end{dfn}

A stronger version of Theorem~\ref{thm:purity_of_nchoosek_mutations} has been shown by Scott:
\begin{thm}[see~\cite{Scott,Scott2}] For $\C \subset {[n] \choose k}$, $\C$ is a maximal weakly separated collection iff the set of Pl\"ucker coordinates $\{\Delta_I\}_{I \in \C}$ is a cluster in the cluster algebra structure on the coordinate ring of $Gr(k,n)$.
\end{thm}

This theorem associates maximal weakly separated collections with clusters, and $k$-tuples in ${[n] \choose k}$ with cluster variables, which are the Pl\"ucker coordinates. Two cluster variables can appear in the same cluster iff they are weakly separated. This trait leads to the following natural question: Given any pair of cluster variables -- how far are they from appearing in the same cluster?
More formally, let $I,J \in {[n] \choose k}$. Define:
\begin{equation}\label{eqd1}
\resizebox{.915\hsize}{!}{$d(I,J)=k(n-k)+1-\max \Big\{|\C_1 \cap \C_2| : \C_1,\C_2 \subset {[n] \choose k},\ I \in \C_1, J \in \C_2  \Big\}$}
\end{equation}
such that both $\C_1$ and $\C_2$ on the right hand side are weakly separated collections. Theorem~\ref{thm:purity_of_nchoosek} implies that $d(I,J)=0$ iff $I$ and $J$ are weakly separated, in which case we can take $\C_1=\C_2$ to be any maximal weakly separated collection containing $I$ and $J$. Thus, $d(\cdot,\cdot)$ measures how close a pair of $k$-tuples is to being weakly separated. This notion can be extended to any cluster algebra, by replacing $k(n-k)+1$ with the rank of the algebra and letting $\C_1$ and $\C_2$ be clusters. This defines a distance between cluster variables. Another (and different) notion of distance between clusters was studied in \cite{Oh}.

For a fixed $n$, we say that $I\subset [n]$ is an \emph{interval} if $I$ is of the form $[a,b]=\{a,a+1,\ldots,b-1, b\}$. If $b<a$ then we consider the elements in $[a,b]$ modulo $n$. For example, if $n=6$ then $[2,4]=\{2,3,4\}$ and $[5,2]=\{5,6,1,2\}$.
Note that $d(\cdot,\cdot)$ does not satisfy the triangle inequality since if $I$ is an interval then $d(I,J)=0$ for any $J$. Therefore both $\C_1$ and $\C_2$ always contain the following $n$ \emph{boundary intervals}:
\[\B_{k,n}=\Big\{\{1,2,\ldots,k\},\{2,3,\ldots,k+1\},\ldots, \{n,1,2,\ldots,k-1\}\Big\},\]
and thus $0 \leq d(I,J) \leq k(n-k)+1-n=(k-1)(n-k-1)$ for any pair $I, J \in {[n]\choose k}$.
\begin{dfn}
For any pair of $k$-tuples $I,J\in {[n]\choose k}$ we define the domain $\A_{I,J}\subset {[n]\choose k}$ to be the collection of all $k$-tuples that are weakly separated from both $I$ and $J$.
\end{dfn}
We can now define $d(\cdot,\cdot)$ equivalently as follows:
\begin{equation}\label{eqd2}
d(I,J)=k(n-k)+1-\max \Big\{|\C| : \C \subset \A_{I,J} \textrm{ is weakly separated } \Big\}.
\end{equation}
The optimal collection $\C$ on the right hand side of (\ref{eqd2}) would be just $\C_1 \cap \C_2$ for the optimal pair $\C_1$ and $\C_2$ in~(\ref{eqd1}). The equivalence of the two definitions follows from Theorem~\ref{thm:purity_of_nchoosek}.
\begin{example}
Let $I=\{1,2,4\}, J=\{3,5,6\} \in {[6] \choose 3}$. Since $I$ and $J$ are not weakly separated,
\[\max \Big\{|\C| \mid \C \subset \A_{I,J} \textrm{ is weakly separated } \Big\}< k(n-k)+1=10.\]
This maximum actually equals $8$ and is achieved at
\[\C=\B_{3,6}\cup\Big\{\{1,2,5\},\{1,3,4\}\Big\}.\]
Therefore,
\[d(I,J)=k(n-k)+1-8=2.\]

Let us consider another pair $\widehat{I}=\{1,3,5\}, \widehat{J}=\{2,4,6\} \in {[6] \choose 3}$. This time, the maximal size of a weakly separated collection in $\A_{\widehat{I},\widehat{J}}$ is $6$ and the only such collection is $\B_{3,6}$.
Hence, $d(\widehat{I},\widehat{J})=(k-1)(n-k-1)=4.$ Note that this value is the largest one that $d(\cdot,\cdot)$ can take for $n=6$ and $k=3$. This example can be generalized to any $k$ with $n=2k$: for $I:=\{1,3,\dots,2k-1\}$, we will see in Lemma~\ref{lemma:characterizationI} that $\A_{I,\bar I}=\B_{k,2k}$ and thus $d(I,\bar I)=(k-1)(n-k-1).$
\end{example}
Thus finding the maximal size of a weakly separated collection in $\A_{I,J}$ enables us to calculate $d(I,J)$. In this paper, we go further and show that for any "generic" pair $I$ and $J$, $\A_{I,J}$ is a pure domain, and we find its rank. This not only gives us the value of $d(I,J)$ for such pairs, but also introduces a new class of pure domains with a different structure from the previously known classes.

\section{Main results}
\label{sec:Mainresult}
In this section we state our main result, which deals with the purity of $\A_{I,J}$ for \emph{balanced} pairs $I,J$. We first discuss the case in which $I$ and $J$ form a complementary pair of sets (so $I \in {[2k] \choose k}$ and $J=\bar I=[2k] \setminus I$) and then proceed to the general case.
\begin{dfn}\label{dfn:intervals_associated_with_A}
Fix any integer $k$ and let $I\in {[2k] \choose k}$ be a set. Then $I$ and its complement $\II$ partition the circle (with numbers $1,2,\ldots,2k$) into an even number of intervals $I=P_1\cup P_3 \cup \ldots \cup P_{2u-1}$, $\II=P_2\cup P_4 \cup \ldots \cup P_{2u}$ for some $u \geq 1$, where $P_i$ are intervals for any $1 \leq i \leq 2u$ and $P_1<P_2<P_3<\dots<P_{2u-1}<P_{2u}$ (these inequalities are ``modulo $2k$''). We say that $\{P_i\}_{i=1}^{2u}$ are the \emph{intervals associated with $I$}. We also let $p_i:=|P_i|$ be their cardinalities and say that $(p_1,p_2,\dots,p_{2u})$ is the \emph{partition of the circle associated with $I$}. We say that $I$ is \emph{balanced} if for any $i\neq j\in [2u]$ we have $p_i+p_j<k$.
\end{dfn}

See Figure~\ref{fig:balanced} for examples of balanced and non-balanced sets.

\begin{figure}
 \centering

\begin{tabular}{|c|c|}\hline
 & \\
\begin{tikzpicture}[scale=0.6,every node/.style={scale=0.6}]
\node[draw,ellipse,white,fill=black] (node1) at (4.00,0.00) {$1$};
\node[draw,ellipse,white,fill=black] (node2) at (3.60,1.74) {$2$};
\node[draw,ellipse,white,fill=black] (node3) at (2.49,3.13) {$3$};
\node[draw,ellipse,black,fill=white] (node4) at (0.89,3.90) {$4$};
\node[draw,ellipse,black,fill=white] (node5) at (-0.89,3.90) {$5$};
\node[draw,ellipse,white,fill=black] (node6) at (-2.49,3.13) {$6$};
\node[draw,ellipse,white,fill=black] (node7) at (-3.60,1.74) {$7$};
\node[draw,ellipse,black,fill=white] (node8) at (-4.00,0.00) {$8$};
\node[draw,ellipse,black,fill=white] (node9) at (-3.60,-1.74) {$9$};
\node[draw,ellipse,black,fill=white] (node10) at (-2.49,-3.13) {$10$};
\node[draw,ellipse,white,fill=black] (node11) at (-0.89,-3.90) {$11$};
\node[draw,ellipse,white,fill=black] (node12) at (0.89,-3.90) {$12$};
\node[draw,ellipse,black,fill=white] (node13) at (2.49,-3.13) {$13$};
\node[draw,ellipse,black,fill=white] (node14) at (3.60,-1.74) {$14$};
\draw[] (0.00,0.00) -- (5.07,-1.16);
\draw[] (0.00,0.00) -- (2.26,4.69);
\draw[] (0.00,0.00) -- (-2.26,4.69);
\draw[] (0.00,0.00) -- (-5.07,1.16);
\draw[] (0.00,0.00) -- (-2.26,-4.69);
\draw[] (0.00,0.00) -- (2.26,-4.69);
\node[anchor=center] (nodeP0) at (1.77,0.98) {\scalebox{1.5}{$p_1=3$}};
\node[anchor=center] (nodeP1) at (0,2.86) {\scalebox{1.5}{$p_2=2$}};
\node[anchor=center] (nodeP2) at (-2.04,1.41) {\scalebox{1.5}{$p_3=2$}};
\node[anchor=center] (nodeP3) at (-1.77,-0.98) {\scalebox{1.5}{$p_4=3$}};
\node[anchor=center] (nodeP4) at (0.10,-2.8) {\scalebox{1.5}{$p_5=2$}};
\node[anchor=center] (nodeP5) at (2.04,-1.41) {\scalebox{1.5}{$p_6=2$}};
\node[anchor=west] (nodeI) at (4.16,4.16) {\scalebox{1.7}{$\bullet\in I$}};
\node[anchor=west] (nodeJ) at (4.16,3.33) {\scalebox{1.7}{$\circ\in \bar I$}};
\end{tikzpicture}
&
\begin{tikzpicture}[scale=0.6,every node/.style={scale=0.6}]
\node[draw,ellipse,white,fill=black] (node1) at (4.00,0.00) {$1$};
\node[draw,ellipse,white,fill=black] (node2) at (3.60,1.74) {$2$};
\node[draw,ellipse,white,fill=black] (node3) at (2.49,3.13) {$3$};
\node[draw,ellipse,black,fill=white] (node4) at (0.89,3.90) {$4$};
\node[draw,ellipse,black,fill=white] (node5) at (-0.89,3.90) {$5$};
\node[draw,ellipse,white,fill=black] (node6) at (-2.49,3.13) {$6$};
\node[draw,ellipse,white,fill=black] (node7) at (-3.60,1.74) {$7$};
\node[draw,ellipse,black,fill=white] (node8) at (-4.00,0.00) {$8$};
\node[draw,ellipse,black,fill=white] (node9) at (-3.60,-1.74) {$9$};
\node[draw,ellipse,black,fill=white] (node10) at (-2.49,-3.13) {$10$};
\node[draw,ellipse,black,fill=white] (node11) at (-0.89,-3.90) {$11$};
\node[draw,ellipse,white,fill=black] (node12) at (0.89,-3.90) {$12$};
\node[draw,ellipse,white,fill=black] (node13) at (2.49,-3.13) {$13$};
\node[draw,ellipse,black,fill=white] (node14) at (3.60,-1.74) {$14$};
\draw[] (0.00,0.00) -- (5.07,-1.16);
\draw[] (0.00,0.00) -- (2.26,4.69);
\draw[] (0.00,0.00) -- (-2.26,4.69);
\draw[] (0.00,0.00) -- (-5.07,1.16);
\draw[] (0.00,0.00) -- (-0.00,-5.20);
\draw[] (0.00,0.00) -- (4.07,-3.24);
\node[anchor=center,color=red] (nodeP0) at (1.77,0.98) {\scalebox{1.5}{$p_1=3$}};
\node[anchor=center] (nodeP1) at (0,2.86) {\scalebox{1.5}{$p_2=2$}};
\node[anchor=center,color=black] (nodeP2) at (-2.04,1.41) {\scalebox{1.5}{$p_3=2$}};
\node[anchor=center,color=red] (nodeP3) at (-1.41,-1.41) {\scalebox{1.5}{$p_4=4$}};
\node[anchor=center,color=black] (nodeP4) at (1.21,-2.24) {\scalebox{1.5}{$p_5=2$}};
\node[anchor=center,color=black] (nodeP5) at (2.60,-1.15) {\scalebox{1.5}{$p_6=1$}};
\node[anchor=west] (nodeI) at (4.16,4.16) {\scalebox{1.7}{$\bullet\in I$}};
\node[anchor=west] (nodeJ) at (4.16,3.33) {\scalebox{1.7}{$\circ\in \bar I$}};
\end{tikzpicture}\\
 & \\
$I=\{1\ 2\ 3,6\ 7,11\ 12\}$ is balanced & $I=\{1\ 2\ 3,6\ 7,12\ 13\}$ is {\bf not} balanced\\
 &   since $3+4=p_1+p_4\geq k=7$ \\\hline
\end{tabular}
\caption{\label{fig:balanced} A balanced set (left) and a non-balanced set (right) in $[14]\choose 7$.}
\end{figure}
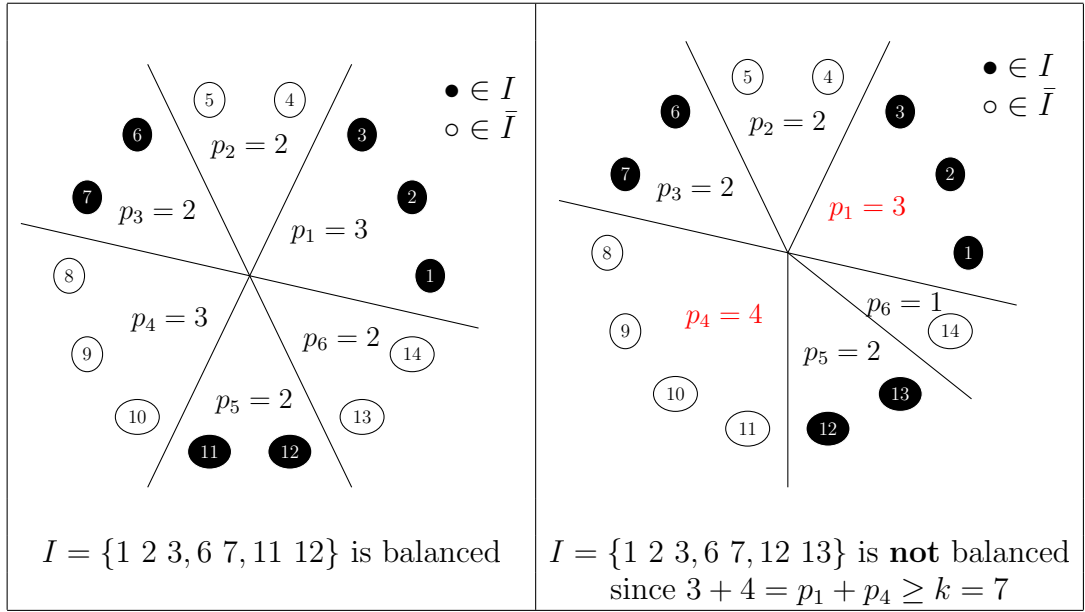

For an even number $t$, we denote by $\P_t$ the collection of all sets $I$ for which $2u=t$ ($2u$ is taken from the definition above). Clearly, $I$ and $\II$ are weakly separated iff $I \in \P_2$. In addition, note that if $I \in \P_4$ then it is not balanced, since $p_1+p_3=p_2+p_4=k$. We study the structure of $\A_{I,\II}$ and $d(I,\II)$ for $I\in \P_4$ in the last section and relate them to the octahedron recurrence. Below is our main result for the complementary case.

\begin{thm}
\label{thm:main}
 Let $I\in\ch$ be balanced and let $(p_1,p_2,\dots,p_{2u})$ be the partition of the circle associated with $I$. Then $\A_{I,\II}$ is a pure domain of rank
\begin{equation}\label{eq:rankofW}
2k+\sum_{i=1}^{2u} {p_i\choose 2}.
\end{equation}
In other words, any maximal (by inclusion) weakly separated collection $\W\subset\A_{I,\II}$ has size given by (\ref{eq:rankofW}).
\end{thm}

\begin{remark}\label{unbalancedremark}
This theorem fails for some non-balanced sets. But note that if the set $I$ is chosen uniformly at random then it is clear that $I$ will be balanced with probability close to $1$ for large values of $k$, so the ``balancedness'' property can be thought of as the analogue of being a ``generic'' set.
\end{remark}

\begin{figure}
 \centering
\scalebox{0.52}{
\begin{tikzpicture}[yscale=0.4,xscale=0.4,every node/.style={scale=0.7}]
\node[draw,ellipse,black,fill=white] (node123456) at (7.50,25.79) {\scalebox{1.3}{$1\,2\,3\,4\,5\,6$}};
\node[draw,ellipse,black,fill=white] (node123457) at (4.50,20.67) {\scalebox{1.3}{$1\,2\,3\,4\,5\,7$}};
\node[draw,ellipse,black,fill=white] (node1234512) at (20.49,23.54) {\scalebox{1.3}{$1\,2\,3\,4\,5\,12$}};
\node[draw,ellipse,black,fill=white] (node12341112) at (27.99,19.65) {\scalebox{1.3}{$1\,2\,3\,4\,11\,12$}};
\node[draw,ellipse,black,fill=white] (node12351112) at (19.24,19.35) {\scalebox{1.3}{$1\,2\,3\,5\,11\,12$}};
\node[draw,ellipse,black,fill=white] (node123101112) at (27.99,15.15) {\scalebox{1.3}{$1\,2\,3\,10\,11\,12$}};
\node[draw,ellipse,black,fill=white] (node129101112) at (20.49,11.25) {\scalebox{1.3}{$1\,2\,9\,10\,11\,12$}};
\node[draw,ellipse,black,fill=white] (node139101112) at (15.75,14.07) {\scalebox{1.3}{$1\,3\,9\,10\,11\,12$}};
\node[draw,ellipse,black,fill=white] (node189101112) at (7.50,9.00) {\scalebox{1.3}{$1\,8\,9\,10\,11\,12$}};
\node[draw,ellipse,black,fill=white] (node234567) at (-7.50,25.79) {\scalebox{1.3}{$2\,3\,4\,5\,6\,7$}};
\node[draw,ellipse,black,fill=white] (node345678) at (-20.49,23.54) {\scalebox{1.3}{$3\,4\,5\,6\,7\,8$}};
\node[draw,ellipse,black,fill=white] (node345679) at (-15.75,20.72) {\scalebox{1.3}{$3\,4\,5\,6\,7\,9$}};
\node[draw,ellipse,black,fill=white] (node3456710) at (-8.00,20.42) {\scalebox{1.3}{$3\,4\,5\,6\,7\,10$}};
\node[draw,ellipse,black,fill=white] (node456789) at (-27.99,19.65) {\scalebox{1.3}{$4\,5\,6\,7\,8\,9$}};
\node[draw,ellipse,black,fill=white] (node4567910) at (-19.50,16.52) {\scalebox{1.3}{$4\,5\,6\,7\,9\,10$}};
\node[draw,ellipse,black,fill=white] (node5678910) at (-27.99,15.15) {\scalebox{1.3}{$5\,6\,7\,8\,9\,10$}};
\node[draw,ellipse,black,fill=white] (node67891011) at (-20.49,11.25) {\scalebox{1.3}{$6\,7\,8\,9\,10\,11$}};
\node[draw,ellipse,black,fill=white] (node67891012) at (-15.75,14.07) {\scalebox{1.3}{$6\,7\,8\,9\,10\,12$}};
\node[draw,ellipse,black,fill=white] (node789101112) at (-7.50,9.00) {\scalebox{1.3}{$7\,8\,9\,10\,11\,12$}};
\fill [opacity=0.2,black] (node123456.center) -- (node123457.center) -- (node234567.center) -- cycle;
\fill [opacity=0.2,black] (node1234512.center) -- (node12341112.center) -- (node12351112.center) -- cycle;
\fill [opacity=0.2,black] (node123101112.center) -- (node129101112.center) -- (node139101112.center) -- cycle;
\fill [opacity=0.2,black] (node345678.center) -- (node345679.center) -- (node456789.center) -- cycle;
\fill [opacity=0.2,black] (node345679.center) -- (node3456710.center) -- (node4567910.center) -- cycle;
\fill [opacity=0.2,black] (node456789.center) -- (node4567910.center) -- (node5678910.center) -- cycle;
\fill [opacity=0.2,black] (node67891011.center) -- (node67891012.center) -- (node789101112.center) -- cycle;
\draw[line width=0.12mm,black] (node123456) -- (node123457);
\draw[line width=0.12mm,black] (node123456) -- (node1234512);
\draw[line width=0.12mm,black] (node123456) -- (node234567);
\draw[line width=0.12mm,black] (node123457) -- (node123456);
\draw[line width=0.12mm,black] (node123457) -- (node1234512);
\draw[line width=0.12mm,black] (node123457) -- (node234567);
\draw[line width=0.12mm,black] (node1234512) -- (node123456);
\draw[line width=0.12mm,black] (node1234512) -- (node123457);
\draw[line width=0.12mm,black] (node1234512) -- (node12341112);
\draw[line width=0.12mm,black] (node1234512) -- (node12351112);
\draw[line width=0.12mm,black] (node12341112) -- (node1234512);
\draw[line width=0.12mm,black] (node12341112) -- (node12351112);
\draw[line width=0.12mm,black] (node12341112) -- (node123101112);
\draw[line width=0.12mm,black] (node12351112) -- (node1234512);
\draw[line width=0.12mm,black] (node12351112) -- (node12341112);
\draw[line width=0.12mm,black] (node12351112) -- (node123101112);
\draw[line width=0.12mm,black] (node123101112) -- (node12341112);
\draw[line width=0.12mm,black] (node123101112) -- (node12351112);
\draw[line width=0.12mm,black] (node123101112) -- (node129101112);
\draw[line width=0.12mm,black] (node123101112) -- (node139101112);
\draw[line width=0.12mm,black] (node129101112) -- (node123101112);
\draw[line width=0.12mm,black] (node129101112) -- (node139101112);
\draw[line width=0.12mm,black] (node129101112) -- (node189101112);
\draw[line width=0.12mm,black] (node139101112) -- (node123101112);
\draw[line width=0.12mm,black] (node139101112) -- (node129101112);
\draw[line width=0.12mm,black] (node139101112) -- (node189101112);
\draw[line width=0.12mm,black] (node189101112) -- (node129101112);
\draw[line width=0.12mm,black] (node189101112) -- (node139101112);
\draw[line width=0.12mm,black] (node189101112) -- (node789101112);
\draw[line width=0.12mm,black] (node234567) -- (node123456);
\draw[line width=0.12mm,black] (node234567) -- (node123457);
\draw[line width=0.12mm,black] (node234567) -- (node345678);
\draw[line width=0.12mm,black] (node345678) -- (node234567);
\draw[line width=0.12mm,black] (node345678) -- (node345679);
\draw[line width=0.12mm,black] (node345678) -- (node456789);
\draw[line width=0.12mm,black] (node345679) -- (node345678);
\draw[line width=0.12mm,black] (node345679) -- (node3456710);
\draw[line width=0.12mm,black] (node345679) -- (node456789);
\draw[line width=0.12mm,black] (node345679) -- (node4567910);
\draw[line width=0.12mm,black] (node3456710) -- (node345679);
\draw[line width=0.12mm,black] (node3456710) -- (node4567910);
\draw[line width=0.12mm,black] (node456789) -- (node345678);
\draw[line width=0.12mm,black] (node456789) -- (node345679);
\draw[line width=0.12mm,black] (node456789) -- (node4567910);
\draw[line width=0.12mm,black] (node456789) -- (node5678910);
\draw[line width=0.12mm,black] (node4567910) -- (node345679);
\draw[line width=0.12mm,black] (node4567910) -- (node3456710);
\draw[line width=0.12mm,black] (node4567910) -- (node456789);
\draw[line width=0.12mm,black] (node4567910) -- (node5678910);
\draw[line width=0.12mm,black] (node5678910) -- (node456789);
\draw[line width=0.12mm,black] (node5678910) -- (node4567910);
\draw[line width=0.12mm,black] (node5678910) -- (node67891011);
\draw[line width=0.12mm,black] (node5678910) -- (node67891012);
\draw[line width=0.12mm,black] (node67891011) -- (node5678910);
\draw[line width=0.12mm,black] (node67891011) -- (node67891012);
\draw[line width=0.12mm,black] (node67891011) -- (node789101112);
\draw[line width=0.12mm,black] (node67891012) -- (node5678910);
\draw[line width=0.12mm,black] (node67891012) -- (node67891011);
\draw[line width=0.12mm,black] (node789101112) -- (node189101112);
\draw[line width=0.12mm,black] (node789101112) -- (node67891011);
\draw[line width=0.12mm,black] (node234567) -- (node3456710);
\draw[line width=0.12mm,black] (node67891012) -- (node789101112);
\node[draw,ellipse,black,fill=white] (node123456) at (7.50,25.79) {\scalebox{1.3}{$1\,2\,3\,4\,5\,6$}};
\node[draw,ellipse,black,fill=white] (node123457) at (4.50,20.67) {\scalebox{1.3}{$1\,2\,3\,4\,5\,7$}};
\node[draw,ellipse,black,fill=white] (node1234512) at (20.49,23.54) {\scalebox{1.3}{$1\,2\,3\,4\,5\,12$}};
\node[draw,ellipse,black,fill=white] (node12341112) at (27.99,19.65) {\scalebox{1.3}{$1\,2\,3\,4\,11\,12$}};
\node[draw,ellipse,black,fill=white] (node12351112) at (19.24,19.35) {\scalebox{1.3}{$1\,2\,3\,5\,11\,12$}};	
\node[draw,ellipse,black,fill=white] (node123101112) at (27.99,15.15) {\scalebox{1.3}{$1\,2\,3\,10\,11\,12$}};
\node[draw,ellipse,black,fill=white] (node129101112) at (20.49,11.25) {\scalebox{1.3}{$1\,2\,9\,10\,11\,12$}};
\node[draw,ellipse,black,fill=white] (node139101112) at (15.75,14.07) {\scalebox{1.3}{$1\,3\,9\,10\,11\,12$}};
\node[draw,ellipse,black,fill=white] (node189101112) at (7.50,9.00) {\scalebox{1.3}{$1\,8\,9\,10\,11\,12$}};
\node[draw,ellipse,black,fill=white] (node234567) at (-7.50,25.79) {\scalebox{1.3}{$2\,3\,4\,5\,6\,7$}};
\node[draw,ellipse,black,fill=white] (node345678) at (-20.49,23.54) {\scalebox{1.3}{$3\,4\,5\,6\,7\,8$}};
\node[draw,ellipse,black,fill=white] (node345679) at (-15.75,20.72) {\scalebox{1.3}{$3\,4\,5\,6\,7\,9$}};
\node[draw,ellipse,black,fill=white] (node3456710) at (-8.00,20.42) {\scalebox{1.3}{$3\,4\,5\,6\,7\,10$}};
\node[draw,ellipse,black,fill=white] (node456789) at (-27.99,19.65) {\scalebox{1.3}{$4\,5\,6\,7\,8\,9$}};
\node[draw,ellipse,black,fill=white] (node4567910) at (-19.50,16.52) {\scalebox{1.3}{$4\,5\,6\,7\,9\,10$}};
\node[draw,ellipse,black,fill=white] (node5678910) at (-27.99,15.15) {\scalebox{1.3}{$5\,6\,7\,8\,9\,10$}};
\node[draw,ellipse,black,fill=white] (node67891011) at (-20.49,11.25) {\scalebox{1.3}{$6\,7\,8\,9\,10\,11$}};
\node[draw,ellipse,black,fill=white] (node67891012) at (-15.75,14.07) {\scalebox{1.3}{$6\,7\,8\,9\,10\,12$}};
\node[draw,ellipse,black,fill=white] (node789101112) at (-7.50,9.00) {\scalebox{1.3}{$7\,8\,9\,10\,11\,12$}};
\end{tikzpicture}
}
\caption{\label{fig:max_collection} A maximal by inclusion weakly separated collection $\W\subset\A_{I,\overline I}$ for $I=\{1,4,5,8,9,10\}\in{[12]\choose 6}$. We have $(p_1,\dots,p_6)=(1,2,2,2,3,2)$. Since $I$ is balanced, the size of this collection equals $12+0+1+1+1+3+1=19$. The white and black triangles form a \emph{plabic tiling} from~\cite{P}. The whole collection lies outside a simple closed polygonal chain $\S$ that we introduce in Section~\ref{section:proof}. This chain $\S$ depends on $\W$, unlike the ones in~\cite{P} and~\cite{DKK10,DKK14}.}
\end{figure}
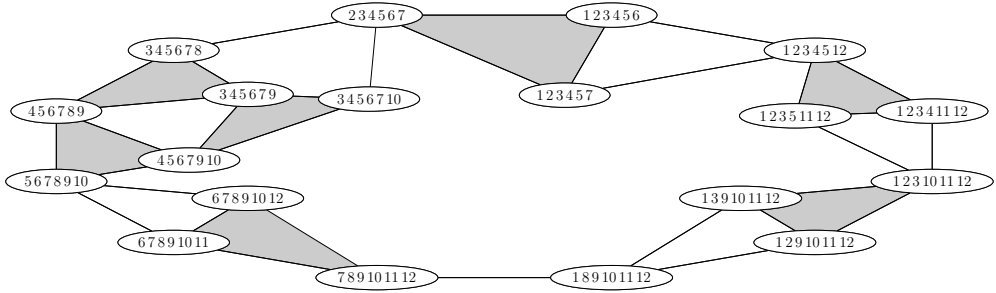


\begin{example}\label{example_I2}
Let
\[I=\{1,2,4,6,8\}, \II=\{3,5,7,9,10\} \in {[10] \choose 5}.\]
Then
\[\A_{I,\II}=\B_{5,10}\cup\Big\{\{1,2,3,4,9\},\{1,3,4,5,6\},\{2,7,8,9,10\},\{5,6,7,8,10\}\Big\}.\]
Consider a graph with vertex set $\A_{I,\II}$ and an edge between $S$ and $T$ iff $S$ and $T$ are weakly separated. Then the vertices labeled by the elements of $\B_{5,10}$ are connected to each other and to all other vertices and the remaining four vertices form the square in Figure~\ref{fig:square}. There are no triangles in Figure~\ref{fig:square}, so there are $4$ maximal weakly separated collections in $\A_{I,\II}$:

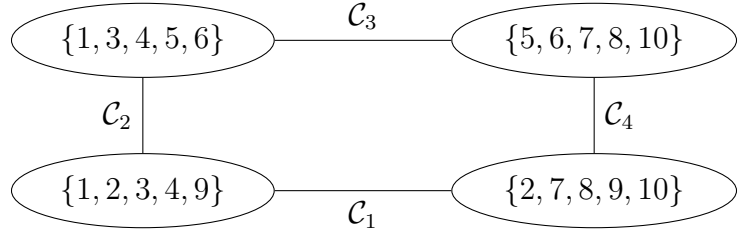
\begin{figure}
 \centering
\begin{tikzpicture}

 \node[draw,ellipse] (bl) at (0,0)  {$\{1,2,3,4,9\}$};
 \node[draw,ellipse] (tl) at (0,2)  {$\{1,3,4,5,6\}$};
 \node[draw,ellipse] (tr) at (6,2)  {$\{5,6,7,8,10\}$};
 \node[draw,ellipse] (br) at (6,0)  {$\{2,7,8,9,10\}$};
 \draw (bl)--(br) node [midway, below] {$\C_1$}  --(tr) node [midway, right] {$\C_4$} --(tl) node [midway, above] {$\C_3$} --(bl) node [midway, left] {$\C_2$};
\end{tikzpicture}
 \caption{\label{fig:square} Elements of $\A_{I,\II}\setminus \B_{5,10}$. Edges correspond to weakly separated pairs.}
\end{figure}

\[\C_1=\B_{5,10}\cup\Big\{\{1,2,3,4,9\},\{2,7,8,9,10\}\Big\},\ \C_2=\B_{5,10}\cup\Big\{\{1,2,3,4,9\},\{1,3,4,5,6\}\Big\},\]
\[\C_3=\B_{5,10}\cup\Big\{\{1,3,4,5,6\},\{5,6,7,8,10\}\Big\},\ \C_4=\B_{5,10}\cup\Big\{\{2,7,8,9,10\},\{5,6,7,8,10\}\Big\}.\]

All of them are of size $12$, and hence $\A_{I,\II}$ is a pure domain of rank $12$, which agrees with the statement of Theorem~\ref{thm:main} since $12=2\cdot 5+ {2 \choose 2}+ {2 \choose 2}$. A more elaborate example for Theorem~\ref{thm:main} is given in Figure~\ref{fig:max_collection}.

\end{example}

Another interesting instance of purity that we have discovered is the following ``left-right purity phenomenon'' which also does not lie inside just one simple closed curve as we have noted in Remark~\ref{rmk:simple_closed_curve}. For a positive integer $n$, denote by $\LR{n}$ the collection of all subsets $I\subset [0,n]:=\{0,1,\dots,n\}$ such that $I$ contains \textit{exactly one} of the elements $0$ and $n$. Then we obtain the following description of maximal weakly separated collections inside $\LR{n}$:

\begin{thm}
\label{thm:LR}
\begin{enumerate}
 \item The domain $\LR{n}$ is a pure domain of rank ${n\choose 2}+n+1=\rk 2^{[n]}$;
 \item\label{item:LR_sequence} For each maximal weakly separated collection $\W\subset\LR{n}$ and for all $m=0,\dots , n-1$ there is a unique set $S_m\subset [n-1]$ of size $m$ such that both $S_m\cup \{0\}$ and $S_m\cup \{n\}$ belong to $\W$. For these sets, we have

 $$\emptyset=S_0\subset S_1\subset\dots\subset S_{n-1}=[n-1].$$
\end{enumerate}
\end{thm}

\begin{remark}
 There is a simple bijection $\phi:\LR{n}\to 2^{[n]}$ that just removes the zero: $\phi(I):=I\setminus \{0\}$ for $I\in \LR{n}$. It is a bijection because $0\in I$ iff $n\not\in \phi(I)$. Moreover, if two sets from $\LR{n}$ were weakly separated then their images are also going to be weakly separated, but the converse is not true. To give a counterexample, consider $n=4$ and take two sets $I,J\in\LR{4}$ defined by $I=\{0,2,3\},\ J=\{1,4\}$. They are not weakly separated, but their images $\phi(I)=\{2,3\}$ and $\phi(J)=\{1,4\}$ are. This is why Theorem \ref{thm:LR} is not a simple consequence of Theorem \ref{thm:purity_of_n}.
\end{remark}

\begin{remark}
 The second part of Theorem \ref{thm:LR} follows directly from \cite[Corollary 3.4]{LZ}. An almost identical but slightly different notion appears in \cite[Definition 6]{DKK10} as a \textit{left-right pair}. However neither of these papers mentions the purity of the corresponding domain.
\end{remark}

Next, we generalize Theorem \ref{thm:main} to the case of not necessarily complementary subsets. Namely, take any two subsets $I,J\in {[n]\choose m}$ and put $k:=|I\setminus J|=|J\setminus I|$. After we ignore the elements from $I\cap J$ and $\overline{J\cup I}$, we get two complementary sets $\It,\Jt\in {[2k]\choose k}$. We say that $I$ and $J$ form a \emph{balanced pair} if the set $\It$ is balanced.

\begin{thm}
\label{thm:main_nc}
 Let $I,J\in{[n]\choose m}$ form a balanced pair and let $\It,\Jt,k$ be as above. Let $(p_1,p_2,\dots,p_{2u})$ be the partition of the circle associated with $\It$. Then $\A_{I,J}$ is a pure domain of rank
 $$m(n-m)-k^2+2k+\sum_{i=1}^{2u} {p_i\choose 2}.$$
\end{thm}
Note that the additional term $m(n-m)-k^2$ is nothing but the difference of ranks $\rk {[n]\choose m}-\rk \ch$. In terms of the distance $d(I,J)$, we have the following

\begin{thm}\label{cor:upperbounddistance}
Let $I,J\in{[n]\choose m}$ form a balanced pair and let $k:=|I\setminus J|=|J\setminus I|$ be as above. Then
\[d(I,J)=1+k^2-2k-\sum_{i=1}^{2u} {p_i\choose 2}.\]
If $I$ and $J$ do not form a balanced pair then
\[d(I,J)\leq 1+k^2-2k-\sum_{i=1}^{2u} {p_i\choose 2}.\]
\end{thm}
The first part of this theorem follows from Theorem~\ref{thm:main_nc}. For the second part, we have an even stronger upper bound, see Theorem~\ref{thm:lowerboununbalanced}.

\section{Further notations and background}\label{section:further}
We denote by $<_i$ the cyclically shifted linear order on $[n]$:

$$i <_i i+1 <_i \ldots <_i n <_i 1 <_i \ldots <_i i-1.$$

Recall that for two sets $A,B\subset [n]$ we write $A<B$ whenever $\max(A)<\min(B)$. In addition, we write $A \prec_i B$ if \[A=\{a_1<_ia_2<_i\ldots<_ia_t\},\quad B=\{b_1<_ib_2<_i\ldots<_ib_r\}\]
with $t \leq r$ and $a_m \leq_i b_m$ for all $1 \leq m \leq t$. By $A\prec B$ we mean $A\prec_1 B$.




%


\subsection{Pure domains inside and outside a simple closed curve}
\label{subsect:dkk}

In this subsection we discuss the approaches of \cite[Section 9]{P} and \cite{DKK14} regarding domains inside and outside simple closed curves. We start with defining a map that appears in both of the papers and justifies the geometric intuition that we are using, for example, while thinking about simple closed curves. Let us fix $n$ vectors $\xi_1,\xi_2, \ldots, \xi_n \in \mathbb{R}^2$ so that the points $(0,\xi_1,\dots,\xi_n)$ are the vertices of a convex $n+1$-gon in clockwise order.
Define:
$$Z_n=\{\lambda_1\xi_1+\ldots+\lambda_n\xi_n \mid 0 \leq \lambda_i \leq 1, i=1,2,\ldots,n\}.$$
We identify a subset $I \subset [n]$ with the point $\sum_{i \in I} \xi_i$ in $Z_n$. Note that if two subsets $I$ and $J$ are weakly separated then the corresponding points are different. Indeed, suppose $|I|\leq |J|$ and $I$ surrounds $J$. The latter implies that there exists a vector $v\in\mathbb{R}^2$ such that for any $i\in I$ and $j\in J$, $\<v,\xi_i\><\<v,\xi_j\>$, where $\<\cdot,\cdot\>$ denotes the standard inner product in $\mathbb{R}^2$. Since $|I|\leq |J|$, we have
\[\sum_{i\in I}\<v,\xi_i\><\sum_{j\in J}\<v,\xi_j\>,\]
and therefore $\sum_{i\in I}\xi_i\neq \sum_{j\in J}\xi_j.$

Now if we have a sequence $\S=(S_0,S_1,\dots,S_r=S_0)$ of subsets, we can always view it as a piecewise linear closed curve $\zeta_{\S}$ obtained by concatenating the line-segments connecting consecutive points $S_{i-1}$ and $S_i$ for $i=1,2\ldots,r$. We will see that if $\S$ satisfies certain properties then the corresponding curve will be simple (i.e. non self-intersecting). Let us consider the easiest example of such a curve:

\begin{dfn} (see \cite{DKK14})
A \textit{simple cyclic pattern} is a sequence $\S=(S_1,S_2,\ldots,S_r=S_0)$ of subsets of $[n]$ such that
\begin{enumerate}
 \item $\S$ is weakly separated;
 \item the sets in $\S$ are pairwise distinct;
 \item $|S_{i-1}\bigtriangleup S_i|=1$ for $i=1,\dots,r$.
\end{enumerate}
For $i\in [r]$, $S_i$ is called a \textit{slope} if $|S_{i-1}| \neq |S_{i+1}|$.
\end{dfn}
Let $\D_{\S}:=\{X \subset [n] \mid X \textrm{ is weakly separated from } \S\}$. We will describe bellow how to write $\D_{\S}$ as a union of two pure domains: $\D_{\S}^\inn$ and $\D_{\S}^\outt$.
For $h=0,1,\ldots,n$, let:
$$\A_h:=\{S_i \mid S_i \textrm{ is a slope and } |S_i|=h\},$$
$$\X_h:=\{X \in \D_{\S} \mid |X|=h \textrm{ and } S_i \prec X \textrm{ for an odd number of } S_i \in \A_h\},$$
$$\Y_h:=\{X \in \D_{\S} \mid |X|=h \textrm{ and } S_i \prec X \textrm{ for an even number of } S_i \in \A_h\}.$$

\begin{dfn} (see \cite{DKK14})
\label{dfn:in_out}
$$\D_{\S}^\inn:=\S \cup \X_0 \cup \X_1 \cup \ldots \cup \X_n,
\D_{\S}^\outt:=\S \cup \Y_0 \cup \Y_1 \cup \ldots \cup \Y_n.$$
\end{dfn}

Note that $$\D_{\S}^\inn \cap \D_{\S}^\outt= \S \textrm{ and } \D_{\S}^\inn \cup \D_{\S}^\outt= \D_{\S}.$$

\begin{thm} (see \cite{DKK14})\label{thm:purity_of_simplecyclic}
For a simple cyclic pattern $\S$, the domains $\D_{\S}^\inn$ and $\D_{\S}^\outt$ are pure. Moreover, every pair $X,Y$ such that $X \in \D_{\S}^\inn$ and $Y\in \D_{\S}^\outt$ is weakly separated.
\end{thm}

From the geometric point of view, the following proposition holds:

\begin{prop} (see \cite{DKK14})\label{prop:geometricdef}
\begin{enumerate}
 \item For a simple cyclic pattern $\S$, the curve $\zeta_{\S}$ is non-self-intersecting,
and therefore it subdivides $Z_n$ into two closed regions $\R_{\S}^\inn$ and $\R_{\S}^\outt$ such that
$\R_{\S}^\inn \cap \R_{\S}^\outt= \zeta_\S \textrm{ and } \R_{\S}^\inn \cup \R_{\S}^\outt= Z_n.$
\item We have $\D_{\S}^\inn=\D_{\S} \cap \R_{\S}^\inn$ and $\D_{\S}^\outt=\D_{\S} \cap \R_{\S}^\outt$.
\end{enumerate}
\end{prop}

Besides simple cyclic patterns we will also need \emph{generalized cyclic patterns}. However, we will only need them for the case when all sets from $\S$ have the same size, so we give simplified versions of a definition and a theorem from \cite{DKK14}.

\begin{dfn}
A \textit{generalized cyclic pattern} is a sequence $\S=(S_1,S_2,\ldots,S_r=S_0)$ of subsets of $[n]$ such that
\begin{enumerate}
 \item $\S$ is weakly separated;
 \item the sets in $\S$ are pairwise distinct;
 \item the sets in $\S$ all have the same size;
 \item $|S_{i-1} \bigtriangleup S_i|=2$.
\end{enumerate}
\end{dfn}

For a generalized cyclic pattern, we define $\D_{\S}$ to be the set of all $X\subset [n]$ that are weakly separated from all elements in $\S$ such that the size of $X$ is the same as the size of the elements in $\S$. This is another restriction on sizes which does not appear in \cite{DKK14}.

Unfortunately, there is no definition of $\D_\S^\inn$ in the literature similar to Definition \ref{dfn:in_out}. Following \cite{DKK14}, in order to define $\D_{\S}^\inn$ and $\D_{\S}^\outt$, we use the geometric construction from above. Part (2) in Proposition~\ref{prop:geometricdef} serves as the definition in this case:

\begin{dfn}
For a generalized cyclic pattern $\S$ satisfying properties (\ref{item:non_self_int_1}) and (\ref{item:non_self_int_2}) below, the curve $\zeta_{\S}$ is non-self-intersecting,
and therefore it subdivides $Z_n$ into two closed regions $\R_{\S}^\inn$ and $\R_{\S}^\outt$ such that
\[\R_{\S}^\inn \cap \R_{\S}^\outt= \zeta_\S \textrm{ and } \R_{\S}^\inn \cup \R_{\S}^\outt= Z_n.\]
In this case we define the domains $\D_{\S}^\inn=\D_{\S} \cap \R_{\S}^\inn$ and $\D_{\S}^\outt=\D_{\S} \cap \R_{\S}^\outt$.
\end{dfn}

\begin{thm}[see \cite{DKK14}]\label{thm:purity_of_generalizedcyclic}
Let $\S$ be a generalized cyclic pattern (with subsets of the same size) satisfying the following two properties:\\
\begin{enumerate}
 \item \label{item:non_self_int_1}$\S$ contains no quadruple $S_{p-1},S_p,S_{q-1},S_q$ such that $\{S_{p-1},S_p\}=\{Xi,Xk\}$ and $\{S_{q-1},S_q\}=\{Xj,Xl\}$, where $i<j<k<l$;
 \item \label{item:non_self_int_2} $\S$ contains no quadruple $S_{p-1},S_p,S_{q-1},S_q$ such that $\{S_{p-1},S_p\}=\{X \setminus i,X \setminus k\}$ and $\{S_{q-1},S_q\}=\{X \setminus j,X \setminus l\}$, where $i<j<k<l$.
\end{enumerate}
Then the domains $\D_{\S}^\inn$ and $\D_{\S}^\outt$ are pure, and every element of $\D_{\S}^\inn$ is weakly separated from any element of $\D_{\S}^\outt$.
\end{thm}

\subsection{Grassmann necklaces and decorated permutations}

In this subsection, all the definitions and results are from \cite{P} and \cite{PBig}.
We now define \textit{Grassmann necklaces} and several objects associated with them. A Grassmann necklace is an important instance of a generalized cyclic pattern, and one reason for that is that for the case of Grassmann necklaces the ranks of $\D_\S^\inn$ and $\D_\S^\outt$ can be calculated explicitly.
\begin{dfn}\label{dfn:dec_perm}
A \textit{Grassmann necklace} is a sequence $\I =(I_1,\ldots,I_n,I_{n+1}=I_1)$ of $k$-element subsets of $[n]$ such that for all $i \in [n]$,
\begin{equation}\label{phi}
I_{i+1}=\left\{
                     \begin{array}{ll}
                       I_i\setminus \{i\} \cup \{j\} \textrm{ for some } j \in [n], & \hbox{if $i \in I_i$;} \\
                       I_i, & \hbox{if $i \notin I_i$}
                     \end{array}
                   \right.
\end{equation}
$\I$ is called \textit{connected} if $I_i \neq I_j$ for $i \neq j$. It is easy to check that every connected Grassmann necklace is a generalized cyclic pattern. Every non-connected Grassmann necklace $\I$ can be transformed into a generalized cyclic pattern denoted $\Ired$ by removing all the adjacent repetitions from $\I$.

\end{dfn}

\begin{remark} \label{rmk:positroid}
It follows also from \cite{P} that the collection $\D_\Ired^\inn$ admits a simpler description:
$$\D_\Ired^\inn=\D_\I\cap \Big\{J \in {[n] \choose k} \mid I_i\prec_i J \textrm{ for all } i\in [n] \Big\}.$$
In other words, $\D_\Ired^\inn$ contains all elements of $\D_\Ired$ that are \emph{inside the positroid} associated with $\I$. For the definition of a positroid, see~\cite[Definition 4.2]{P}.
\end{remark}

It turns out that a convenient way of encoding a Grassmann necklace is to use \textit{decorated permutations}.

\begin{dfn}\label{dfn:decorated_perm}
A \emph{decorated permutation} $\pii=(\pi,\col)$ is a permutation $\pi \in \Sfrak_n$ together with a coloring function $\col$ from the set of fixed points $\{i \mid \pi(i) = i \}$ to  $\{1,-1\}$. For $i,j \in [n]$, $\{i,j\}$ forms an \textit{alignment in} $\pi$ if $i,\pi(i),\pi(j),j$ are cyclically ordered (and all distinct). The number of alignments in $\pi$ is denoted by $\Alignments(\pi)$, and the \textit{length} $\ell(\pii)$ is defined to be $k(n-k)-\Alignments(\pi)$.
\end{dfn}
We now describe a bijection between Grassmann necklaces and decorated permutations. Given a Grassmann necklace $\I$, define $\pii_\I=(\pi_\I,\col_\I)$ as follows:
\begin{itemize}
  \item If $I_{i+1}=I_i\setminus \{i\} \cup \{j\}$ for $j \neq i$ then $\pi_\I(i)=j$.
  \item If $I_{i+1}=I_i$ and $i \notin I_i$ (resp., $i \in I_i$) then $\pi_\I(i)=i$ and $\col_\I(i)=1$ (resp., $\col_\I(i)=-1$).
\end{itemize}
We refer the reader to \cite{P} for the construction of the inverse of this map.


We define $\ell(\I)$ to be $\ell(\pii_\I)$, where $\pii_\I$ is the associated decorated permutation of $\I$.
\begin{thm}[see \cite{P}]\label{thm:weak_separation_necklace}
Fix any Grassmann necklace $\I$. Every maximal weakly separated
collection in $\D_{\I}^\inn$ has cardinality $\ell(\I)+1$. Any two maximal weakly separated
collections in $\D_{\I}^\inn$ are linked by a sequence of mutations.
\end{thm}
Note that Theorem~\ref{thm:purity_of_nchoosek} is a special case of the theorem above, by setting $I_i=\{i,i+1,\ldots,i+k-1\} \subset [n]$ for all $i$ (the entries are taken modulo $n$).

\subsection{Plabic graphs}
Another reason for the importance of Grassmann necklaces is that there is an especially nice geometric intuition that helps to understand the structure of the corresponding weakly separated collections. In this subsection, we continue citing the results from \cite{P} and \cite{PBig}.

\begin{dfn}
A \textit{plabic graph} (planar bicolored graph) is a planar undirected graph $G$ drawn inside a disk with vertices colored in black or white colors. The vertices on the boundary of the disk, called the \emph{boundary vertices}, are labeled in clockwise order by the elements of $[n]$.
\end{dfn}
\begin{dfn}
A \textit{strand} in a plabic graph $G$ is a directed path that satisfies the ``rules of the road'': at every black vertex it makes a sharp right turn, and at every white vertex it makes a sharp left turn.
\end{dfn}
\begin{dfn}(see \cite[Theorem 13.2]{PBig}) \label{dfn:reduced}
A plabic graph $G$ is called \textit{reduced} if the following holds:
\begin{itemize}
  \item A strand cannot be a closed loop in the interior of $G$.
  \item If a strand passes through the same edge twice then it must be a simple loop that starts and ends at a boundary leaf.
  \item Given any two strands, if they have two edges $e$ and $e'$ in common then one strand should be directed from $e$ to $e'$ while the other strand should be directed from $e'$ to $e$.
\end{itemize}
\end{dfn}
Any strand in a reduced plabic graph $G$ connects two boundary vertices. We associate a decorated permutation (also called \textit{strand permutation}) $\pii_G=(\pi_G,\col_G)$ with $G$ for which $\pi_G(j)=i$ if the strand that starts at a boundary vertex $j$ ends at a boundary vertex $i$. We say that such strand \emph{is labeled by} $i$. If $\pi_G(i)=i$ then $i$ must be connected to a boundary leaf $v$, and $\col(i)=+1$ if $v$ is white and $\col(i)=-1$ if $v$ is black.

Let us now describe three types of moves on a plabic graph that preserve its decorated permutation:

(M1) Pick a square with trivalent vertices alternating in colors as in Figure~\ref{move1}. Then we can switch the colors of all the vertices.\\
\begin{figure}[h!]
\centering
\includegraphics[height=0.6in]{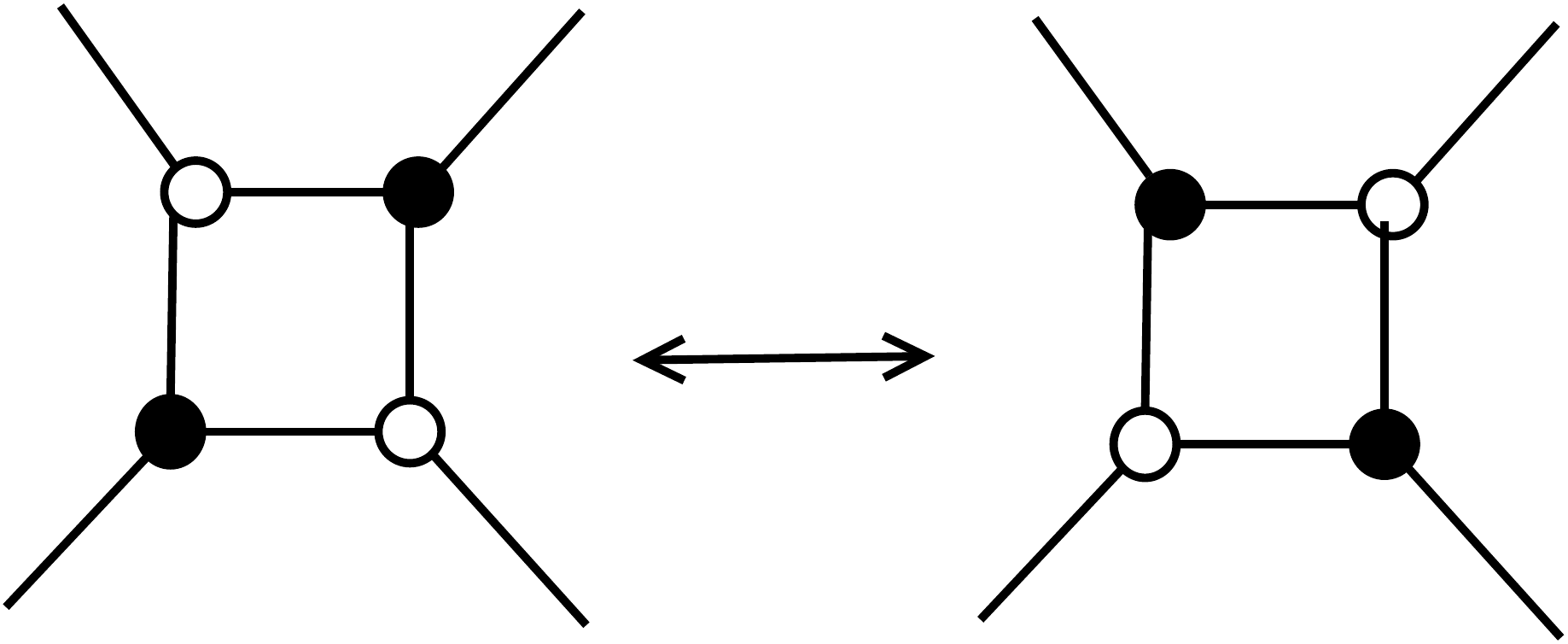}
\caption{(M1) Square move}
\label{move1}
\end{figure}

(M2) Given two adjoint vertices of the same color, we can contract them into one vertex as in Figure~\ref{move2}.\\
\begin{figure}[h!]
\centering
\includegraphics[height=0.4in]{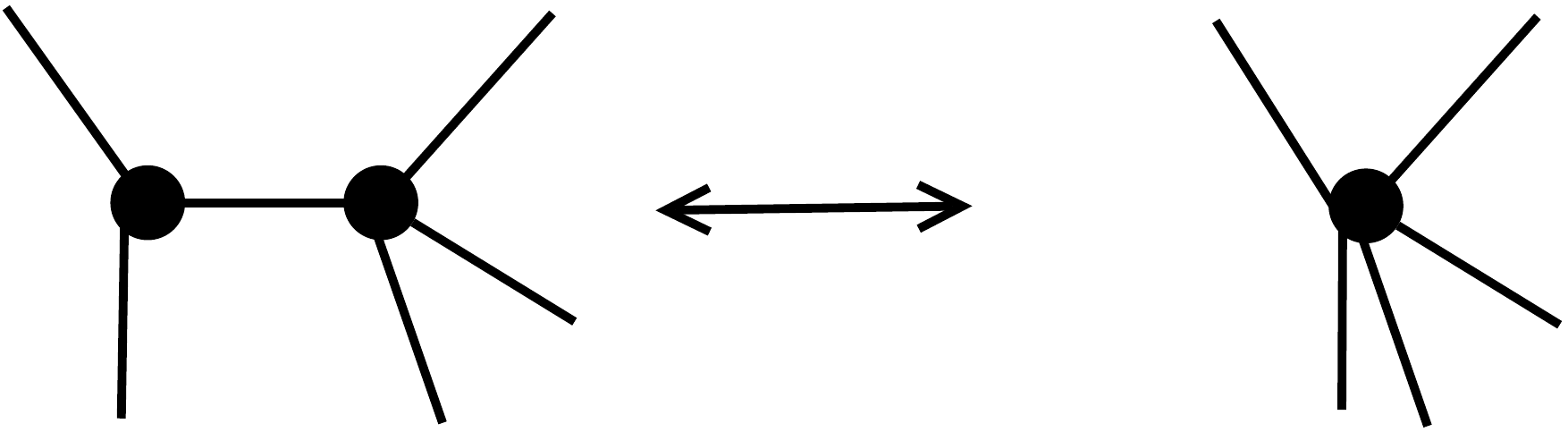}
\caption{(M2)}
\label{move2}
\end{figure}

(M3) We can insert or remove a vertex inside an edge. See Figure~\ref{move3}.
\begin{figure}[h!]
\centering
\includegraphics[height=0.08in]{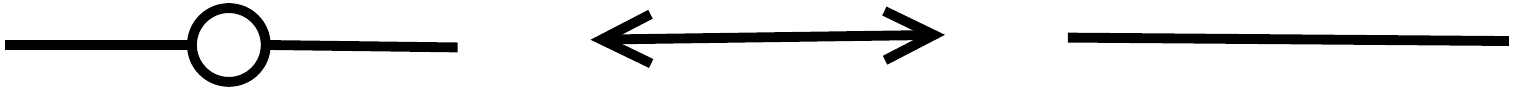}
\caption{(M3)}
\label{move3}
\end{figure}

Figures \ref{move1},\ref{move2}, and \ref{move3} are taken from \cite{PBig}.

\begin{thm}[see \cite{PBig}]
Let $G$ and $G'$ be two reduced plabic graphs with the same number of boundary vertices. Then $\pi_G=\pi_{G'}$ if and only if $G'$ can be obtained from $G$ by a sequence of moves (M1)–(M3).
\end{thm}
We conclude this subsection with a theorem from~\cite{P} that describes the relation between reduced plabic graphs and weakly separated collections. We first describe a certain labeling of the faces by subsets of $[n]$. Given a reduced plabic graph $G$, we place $i$ inside every face $F$ that appears to the left of the strand $i$. We apply this process for every $i \in [n]$, and the label of $F$ is the set of all $i$'s placed inside $F$. We denote by $\F(G)$ the collection of labels that occur on each face of the graph $G$. It was shown in \cite{PBig} that all the faces in $G$ are labeled by the same number of strands.
\begin{thm}[see \cite{P}]
For a decorated permutation $\pii_\I$ and the corresponding Grassmann necklace $\I$, a collection $\CC$ is a maximal weakly separated collection inside $\D_\I$ if and only if it has the form $\CC=\F(G)$ for a reduced plabic graph
$G$ with strand permutation $\pi_\I$. In particular, a maximal weakly separated collection $\CC$ in ${[n] \choose k}$ has the form $\CC=\F(G)$ for a reduced plabic graph $G$ with strand permutation
\begin{equation}\label{eq:taunk}
\pi(i)=i+k\pmod n,\quad i=1,2,\ldots,n.
\end{equation}
\end{thm}

\subsection{The canonical decorated permutation associated with $\A_{A,\AA}$}
Recall that for a set $A \in {[2k] \choose k}$, $\AA$ denotes the complement of $A$. We can always cyclically shift $A$ and $\AA$ in order to have $1 \in A, 2k \in \AA$. In such a case $A$ and $\AA$ are of the form $A=P_1\cup P_3 \cup \ldots \cup P_{2u-1}$, $\AA=P_2\cup P_4 \cup \ldots \cup P_{2u}$ for some $u \geq 1$, where $P_i$ are intervals for any $1 \leq i \leq 2u$ and $P_1<P_2<P_3<\dots<P_{2u-1}<P_{2u}$.


Our running example is going to be $A=\{1,2,3,7,8\} \in {[10] \choose 5}$. Equivalently, $A=[1,3] \cup [7,8]$ and $\AA=[4,6] \cup [9,10]$. Therefore, $\{[1,3],[4,6],[7,8],[9,10]\}$ are the intervals associated with $A$ (see Definition~\ref{dfn:intervals_associated_with_A}) and their lengths are $p_1=p_2=3, p_3=p_4=2$.

For each $k\leq n$ we let $\tau_{k,n}$ be the permutation defined by~\eqref{eq:taunk}. We view permutations as maps $[n]\to [n]$ so if $\sigma$ and $\pi$ are two permutations of $[n]$ then $(\sigma\circ\pi)(i)=\sigma(\pi(i))$. Finally, we write each permutation in one-line notation as follows:
\[\sigma=\left(\sigma(1),\sigma(2),\dots,\sigma(n)\right).\]

Recall that $A$ is called \textit{balanced} if $p_i+p_j<k$ for all $1 \leq i\neq j \leq 2u$ and that by $\A_{A,\AA}$ we denote the collection of all subsets in $\ch$ that are weakly separated from both $A$ and $\AA$. For example, the set $A=\{1,2,3,7,8\}$ is not balanced because $p_1+p_2=6$ and $6$ is not strictly less than $5$. 

By $\tau_A$ we denote the following permutation:
\[\tau_A=\left(p_1,p_1-1,\dots,1,p_1+p_2,p_1+p_2-1,\dots,p_1+1,\dots,2k,2k-1,\dots,2k-p_{2u}+1\right).\]

We say that the permutation $\tau_A\circ\tau_{k,2k}$ is the \emph{canonical decorated permutation associated with} $\A_{A,\AA}$. For $A=\{1,2,3,7,8\}$ from the example above, we have $k=5$ so the permutation $\tau_{5,10}$ is
\begin{eqnarray*}
\tau_{5,10}&=&[6,7,8,9,10,1,2,3,4,5]
\end{eqnarray*}

in one-line notation. Similarly,
\begin{eqnarray*}
\tau_A&=&[\underline{3,2,1},\underline{6,5,4},\underline{8,7},\underline{10,9}].
\end{eqnarray*}

After taking the composition, we get

\begin{eqnarray*}
\tau_A\circ\tau_{5,10}&=&[\underline{4},\underline{8,7},\underline{10,9},\underline{3,2,1},\underline{6,5}].
\end{eqnarray*}

We denote by $\I(\tau_A\circ\tau_{k,2k})$ the Grassmann necklace that corresponds to $\tau_A\circ\tau_{k,2k}$ and call it the \textit{canonical Grassmann necklace associated with} $A$. For $A=\{1,2,3,7,8\}$ as above, its canonical Grassmann necklace $\I(\tau_A\circ\tau_{k,2k})$ is the sequence of sets written in the rows of the table in Figure~\ref{fig:table}.

\begin{figure}
 \centering

\[\begin{array}{|ccc|ccc|cc|cc|}\hline
1&2&3& &5&6& & & &  \\
 &2&3&4&5&6& & & &  \\
 & &3&4&5&6& &8& &  \\
 & & &4&5&6&7&8& &  \\
 & & & &5&6&7&8& &10\\
 & & & & &6&7&8&9&10\\
 & &3& & & &7&8&9&10\\
 &2&3& & & & &8&9&10\\
1&2&3& & & & & &9&10\\
1&2&3& & &6& & & &10\\
1&2&3& &5&6& & & &  \\  \hline
\end{array}\]
 \caption{\label{fig:table} The rows of the table are the sets in the Grassmann necklace $\I(\tau_{\{1,2,3,5,6\}}\circ\tau_{5,10})$.}
\end{figure}


\begin{remark}\label{rmk:length}
Note that $\{i,j\}$ is an alignment (see Definition~\ref{dfn:decorated_perm}) in $\tau_A\circ\tau_{k,2k}$ iff both $\tau_{k,2k}(i)$ and $\tau_{k,2k}(j)$ belong to the same set $P_m$ for some $1 \leq m \leq 2u$. Therefore, \[\ell(\I(\tau_A\circ\tau_{k,2k}))=k^2-\sum_{i=1}^{2u} {p_i \choose 2}.\]
\end{remark}

\section{Proof of Theorem~\ref{thm:LR}}\label{section:proofLR}

Recall that $\LR{n}\subset 2^{[0,n]}$ is the collection of all subsets $S$ of $[0,n]$ such that $|S\cap \{0,n\}|=1$.

Let $\W\subset \LR{n}$ be any weakly separated collection. We call sets of the form
\[\Big\{L\subset [n-1]\mid  L\cup \{0\}\in \W \Big\}\]
\textit{left subsets of }$\W$, while sets of the form
\[\Big\{R\subset [n-1]\mid  R\cup \{n\}\in \W \Big\}\]
will be called \textit{right subsets of }$\W$. For $A, B \subset [ n-1 ]$, we say that $A$ is \textit{to the right} (resp., \textit{to the left}) of $B$ if $B \setminus A < A \setminus B$ (resp., $A \setminus B < B \setminus A $).

We start with the proof of the second part: we will construct inductively a sequence of sets $$\emptyset=S_0\subset S_1\subset\dots\subset S_{n-1}=[n-1]$$ so that $|S_m|=m$ and $S_m\cup \{0\}$ and $S_m\cup\{ n\}$ are weakly separated from $\W$ for any $0<m<n$. In such a case, if $\W$ is maximal then $S_m\cup \{0\}, S_m\cup\{ n\} \in \W$. The base case holds as $S_0$ clearly satisfies these requirements. Suppose we have a sequence $S_0\subset S_1\subset \dots\subset S_m$ satisfying the desired properties. By the inductive hypothesis, $S_m \cup \{ n\}$ is weakly separated from $L \cup \{0\}$ for any left subset $L$, which implies that $S_m$ is to the right of all left subsets. Therefore, there exists an element $j'\in [n-1]\setminus S_m$ such that $S_m\cup\{j'\}$ is to the right of all left subsets (we can just take $j'$ to be the maximal element of $[n-1]\setminus S_m$). Let $j$ be the minimal element such that $S_{m+1}:=S_m\cup\{j\}$ is still to the right of all left subsets in $\W$.

In order to prove the inductive step, we will show that the following holds for any $L$ and $R$, where $L$ is a left subset in $\W$ and $R$ is a right subset in $\W$:
\begin{enumerate}[label=(\alph*)]
    \item \label{a} $S_{m+1} \cup \{ n\}$ is weakly separated from $L \cup \{0\}$.
    \item \label{b} $S_{m+1} \cup \{0\}$ is weakly separated from $L \cup \{0\}$.
    \item \label{c} $S_{m+1} \cup \{0\}$ is weakly separated from $R \cup \{ n\}$.
    \item \label{d} $S_{m+1} \cup \{ n\}$ is weakly separated from $R \cup \{ n\}$.
\end{enumerate}
The parts \ref{a} and \ref{b} hold since $S_{m+1}$ is to the right of $L$.

In order to verify \ref{c} and \ref{d}, we need to show that
$S_{m+1}$ is to the left $R$. Assume to the contrary that $R$ is a right subset that is not to the right of $S_{m+1}$. By the induction hypothesis, $R$ was to the right of $S_m$ (since $S_m \cup \{0\}$ and $R \cup \{ n\}$ were weakly separated). It means that $j\notin R$ and that there is an element $i<j$ that belongs to $R\setminus S_m$. Let us choose the maximal such $i$ so that $(R\setminus S_{m})\cap (i,j]=\emptyset$.

We chose $j$ to be the minimal element so that $S_m\cup\{j\}$ is to the right of all left subsets, which means that we did not choose $i$ for some reason. Thus there is a left subset $L$ such that $i\notin L$ and $L\setminus S_m$ contains some element $k>i$. If $k>j$ then we would not choose $j$, so $k\leq j$. Recall that $(R\setminus S_{m})\cap (i,j]=\emptyset$ so $k\notin R$. We get a contradiction: $L$ should be to the left of $R$ (since $L \cup \{0\}$ and $R \cup \{ n\}$ are weakly separated) but $$i<k,i\in R,i\notin L,k\notin R,k\in L.$$ Therefore our assumption was wrong and $S_{m+1}$ is to the left of all right subsets. This implies that \ref{c} and \ref{d} hold and thus finishes the induction argument. In order to finish the proof of the second part of Theorem~\ref{thm:LR}, the only thing left to show is that if $\W$ is maximal then $S_m$ is unique for all $m$. Assume in contradiction that for some $0 \leq m \leq n$, $S_m$ is not unique. Therefore, there exist two different sets $A, B \subset [n-1]$ such that $|A|=|B|$, and
\[(A \cup \{0\}),(B \cup \{0\}),(A \cup \{ n\}),(B \cup \{ n\}) \in \W.\]
Since $A \cup \{0\}$ and $B \cup \{ n\}$ are weakly separated, $A$ must be to the left of $B$. However, the same argument for $B \cup \{0\}$ and $A \cup \{ n\}$ implies that $B$ is to the left of $A$. This implies that $A=B$, a contradiction. This finishes the proof of the second part of Theorem~\ref{thm:LR}.

Now to prove the first part, let $\W\subset \LR{n}$ be a maximal weakly separated collection, and let $S_0\subset S_1\subset\dots\subset S_{n-1}$ be the unique sequence of sets obtained in the second part. Consider the following simple cyclic pattern inside $[ n ]$:
\begin{eqnarray*}
 \P:=(S_0,(\{0\}\cup S_0),(\{0\}\cup S_1),&\dots&,(\{0\}\cup S_{n-1}),(\{0, n \}\cup S_{n-1}), \\
 (\{n\} \cup S_{n-1}), (\{n\} \cup S_{n-1}),&\dots&, (\{n\} \cup S_0)).
\end{eqnarray*}

First, let us determine which subsets of $[0,n]$ are contained in $\D_\P$. We consider four cases depending on whether a subset contains $0$ and/or $ n $:
\[(\{0\}\cup L), (R\cup \{ n\}), (\{0\}\cup T\cup \{ n\}), T,\]
where $L,R,T \subset [n-1]$.
\begin{enumerate}
 \item \label{type1} $(\{0\}\cup L) \in \D_\P$. In this case $L$ should be to the left of $S_m$ for all $m$ because $\{0\}\cup L$ should be weakly separated from $S_m\cup \{ n\}$. It is easy to check that this criterion is also sufficient in order to have $(\{0\}\cup L) \in \D_\P$.
 \item \label{type2} $(R\cup \{ n\}) \in \D_\P$. In this case $R$ should be to the right of $S_m$ for all $m$ because $R\cup \{ n\}$ should be weakly separated from $S_m\cup \{0\}$. This criterion is also sufficient.
 \item \label{type3} $(\{0\}\cup T\cup \{ n\}) \in \D_\P$. Let $k=|T|$ and consider the set $S_k$. If $T\neq S_k$ then there are elements $t,s\in [n-1]$ such that $t\in T\setminus S_k$ and $s\in S_k\setminus T$. If $t<s$ then we have
  \[|\{0\}\cup T\cup \{ n\}|=k+2>k+1=|\{0\}\cup S_k|\]
 and
 \[t, n \in \left(\{0\}\cup T\cup \{ n\}\right)\setminus \left(\{0\}\cup S_k\right),\]
  while
 \[s\in \left(\{0\}\cup S_k\right) \setminus \left(\{0\}\cup T\cup \{ n\}\right).\]
 In other words, $\{0\}\cup S_k$ has smaller size than $\{0\}\cup T\cup \{ n\}$ and does not surround $\{0\}\cup T\cup \{ n\}$, so they are not weakly separated. But in order to have $\{0\}\cup T\cup \{ n\} \in \D_\P$, it must be weakly separated from $\{0\}\cup S_k\in \P$, a contradiction.

 If $s<t$ then we apply a similar argument, only now we consider $S_k \cup \{ n\}$ instead of $\{0\}\cup S_k$. Thus if $(\{0\}\cup T\cup \{ n\}) \in \D_\P$ then $T=S_k$. It is clear that $(\{0\}\cup S_k\cup \{ n\}) \in \D_\P$.
 \item \label{type4} $T \in \D_\P$. Similarly to the case (\ref{type3}) above, this happens if and only if $T=S_k$.
\end{enumerate}

We now determine the domains $\D_\P^\inn$ and $\D_\P^\outt$. For $1 \leq h \leq n$ the slopes of $\P$ of size $h$ are
\[\A_h:=\Big\{\{0\}\cup S_{h-1},S_{h-1}\cup \{ n\}\Big\}.\]
Recall that
\begin{eqnarray*}
\X_h&:=&\{X \in \D_{\P} \mid |X|=h \textrm{ and } S \prec X \textrm{ for an odd number of } S \in \A_h\},\\
\Y_h&:=&\{X \in \D_{\P} \mid |X|=h \textrm{ and } S \prec X \textrm{ for an even number of } S \in \A_h\}.
\end{eqnarray*}
We would like to show that for each $h$, the only two elements of $\D_\P$ that belong to $\X_h$ are $S_{h}$ (if $h<n$) and $\{0\} \cup S_{h-2}\cup \{ n\}$ (if $h>1$), that is, the subsets of types (\ref{type3}) and (\ref{type4}). Take any subset $X\in \D_\P$ with $|X|=h$. The only two slopes of size $h$ are $\{0\}\cup S_{h-1}$ and $S_{h-1}\cup \{ n\}$. If $X$ is of type (\ref{type1}) (resp., (\ref{type2})) meaning that $X=\{0\}\cup L$ (resp., $X=R\cup \{n\}$) then $X$ is to the left (resp., to the right) of both of the slopes in $\A_h$. Therefore in these two cases $X$ belongs to $\Y_h$. If $X$ is of types (\ref{type3}) or (\ref{type4}) then the slope $\{0\}\cup S_{h-1}$ is to the left of $X$ while the slope $S_{h-1}\cup \{ n\}$ is to the right of $X$, so $X\in \X_h$. To sum up,

\[\X_h:=\Big\{S_{h}\textrm{, if $h<n$}\Big\}\cup\Big\{\{0\} \cup S_{h-2}\cup \{ n\}\textrm{, if $h>1$}\Big\};\]
\begin{eqnarray*}
\Y_h&:=& \Big\{\{0\}\cup L \mid |L|=h-1, L \textrm{ is of type (\ref{type1})}\Big\}\cup\\
    &  & \Big\{R \cup \{n\}\mid |R|=h-1, R \textrm{ is of type (\ref{type2})}\Big\}.
\end{eqnarray*}

Therefore $\D_\P^\inn$ consists of $\P$ and all the elements of type (\ref{type3}) and (\ref{type4}), while $\D_\P^\outt$ consists of $\P$ and all the elements of the form (\ref{type1}) and (\ref{type2}). By Theorem~\ref{thm:purity_of_simplecyclic}, both of these domains are pure. Note that $\W$ is a maximal weakly separated collection inside $\LR{n}$ so $\W \cup \{\emptyset,[0,n ]\}$ is a maximal weakly separated collection inside $\D_\P^\outt$. Hence in order to establish the purity of $\LR{n}$ and to show that its rank is ${n\choose 2}+n+1$, we need to prove that
\[\rk\D_\P^\outt={n\choose 2}+n+3.\]

Let us first analyze maximal weakly separated collections in $\D_\P^\inn$. Note that a pair of the form $S_m$ and $\{0\} \cup S_{m-1}\cup \{ n\}$, is not weakly separated. Therefore the collection $\{S_1,S_2\dots,S_{n-1}\} \cup \P$ is a maximal weakly separated collection inside $\D_\P^\inn$, so $\rk\D_\P^\inn= n-1+|\P|$.
Finally, by Theorems~\ref{thm:purity_of_n} and~\ref{thm:purity_of_simplecyclic}, we have the following equation:

$$\rk\D_\P^\inn+\rk\D_\P^\outt-|\P|=\rk 2^{[0, n ]}={ n+1 \choose 2}+(n+1)+1.$$
Let us substitute $\rk\D_\P^\inn=n-1+|\P|$ and make the cancellations:
$$\rk\D_\P^\outt={ n+1 \choose 2}+3={n\choose 2}+n+3.$$
Unlike $\P$, this number does not depend on the sequence $S_0,\dots,S_{n-1}$, so $\LR{n}$ is pure and its rank is ${n\choose 2}+n+1$. This finishes the proof of Theorem~\ref{thm:LR}.\qed

\section{Isomorphic generalized necklaces}\label{sect:iso}

\def\T{{\mathcal{T}}}
\def\CC{{\mathcal{C}}}

\def\id{{\operatorname{id}}}

Recall from Section \ref{subsect:dkk} that given a Grassmann necklace $\I=(I_1,I_2,\dots,I_n)$ we denote by $\Ired$ a generalized cyclic pattern obtained from $\I$ by removing all adjacent repetitions: we remove $I_k$ from $\I$ iff $I_k=I_{k-1}$. We call a generalized cyclic pattern of the form $\Ired$ a \textit{reduced Grassmann necklace}. These reduced Grassmann neckalces form a nice class of generalized cyclic patterns for which we know the ranks of $\D^\inn_\Ired$ and $\D^\outt_\Ired$. In this section we want to generalize this class to those generalized cyclic patterns which differ from a reduced Grassmann necklace by a simple relabeling of elements of $[n]$.

\begin{dfn}\label{dfn:gr_like_necklace}
Let $\S=\left(S_0,S_1,\dots,S_r=S_0\right)$ be a generalized cyclic pattern satisfying the following properties:
\begin{enumerate}

 \item\label{dfn:i_t_and_j_t}  for all $0\leq t<r$, there exist numbers $i_t\neq j_t$ such that $S_t\setminus S_{t+1}=\{i_t\}$ and $S_{t+1}\setminus S_{t}=\{j_t\}$;
 \item \label{dfn:pairwise_distinct} the numbers $i_1,\dots,i_r$ are pairwise distinct and the numbers $j_1,\dots,j_r$ are pairwise distinct;
 \item \label{dfn:same_sets} $\{i_1,\dots,i_r\}=\{j_1,\dots,j_r\}$.
\end{enumerate}
In this case $\S$ is called \textit{a Grassmann-like necklace}.
\end{dfn}

Note that the property (\ref{dfn:same_sets}) follows from (\ref{dfn:i_t_and_j_t}) and (\ref{dfn:pairwise_distinct}). Property (\ref{dfn:same_sets}) allows us to denote $\act(\S):=\{i_1,\dots,i_r\}=\{j_1,\dots,j_r\}$.

\begin{dfn}\label{dfn:iso}
 Two generalized cyclic patterns $\S=\left(S_0,S_1,\dots,S_r=S_0\right)$ and $\CC=\left(C_0,C_1,\dots,C_r=C_0\right)$ with $S_i,C_i\subset [n]$ are called \textit{isomorphic} if there exists a permutation $\gamma\in \mathfrak{S}_n$ such that:
 \begin{itemize}
  \item $\gamma(C_i)=S_i$ for $i=0\dots r$;
  \item $\gamma$ is a bijection $\D_\CC^\inn\to\D_\S^\inn$;
  \item two subsets $A,B\in\D_\CC^\inn$ are weakly separated if and only if $\gamma(A)$ and $\gamma(B)$ are weakly separated.
 \end{itemize}
 Such a permutation $\gamma$ is called \textit{an isomoprhism} between $\S$ and $\CC$.
\end{dfn}

Clearly, if $\S$ and $\CC$ are isomorphic then the ranks of $\D_\S^\inn$ and $\D_\CC^\inn$ are equal.

Given a Grassmann-like necklace $\S$, we can associate with it two (decorated) permutations $\sigma^{:}=(\sigma,\col_\S)$ and $\pii=(\pi,\col_\S)$ in $\Sfrak_n$ as follows. If we order the elements of $\act(\S)$ in increasing order
\[\act(\S)=\{q_1<q_2<\ldots<q_r\}\]
then for all $t\in [r]$ we put
\[\sigma(q_t)=i_t, \pi(q_t)=j_t.\]
For $i\notin \act(\S)$, we leave $\sigma(i)=\pi(i)=i$ with $\col_\S(i)=+1$ (resp., $\col_\S(i)=-1$) if $i\in S_k$ (resp., $i\notin S_k$) for all $k\in[r]$.

Note that if $\sigma=\id$ and $\act(\S)=[n]$ then $\S$ is just a connected Grassmann necklace. For a (decorated) permutation $(\gamma,\col_\S)$, denote the corresponding reduced Grassmann necklace by $\CC^\circ(\gamma,\col_\S)$ (see the bijection after Definition~\ref{dfn:dec_perm}).

\begin{thm}\label{thm:iso}
 Let $\S$ be a Grassmann-like necklace. Then $\sigma^{-1}$ is an isomorphism between $\S$ and the reduced Grassmann necklace $\CC^\circ(\sp,\col_\S)$.
\end{thm}
\begin{proof}

\def\SS{{\widetilde{\S}}}

\def\Q{{\mathcal{Q}}}
\def\W{{\mathcal{W}}}
\def\sinv{{\sigma^{-1}}}

 Let $\SS$ be a maximal weakly separated collection of subsets from $[r]\choose k$ so that $\S\subset\SS$. Consider the reduced plabic graph $G$ corresponding to $\SS$. It has some faces labeled by the sets from $\S$, and these faces form a \textit{simple closed curve} (see \cite{DKK14}); in particular, the face labeled by $S_t$ and the face labeled by $S_{t+1}$ share either a vertex or an edge for each $t\in [r]$. By uncontracting vertices into edges we get $r$ edges $e_0,\dots,e_{r-1}$ such that $e_t$ lies between faces labeled by $S_t$ and $S_{t+1}$. Let $O$ denote the (topological) circle  that passes through the midpoints of $e_0,\dots,e_{r-1}$ and does not cross other edges of $G$. We have that the strand labeled by $i_t$ enters the circle $O$ through the edge $e_t$ while the strand labeled by $j_t$ exits the circle $O$ through this edge. The faces outside $O$ belong to $\D_\S^\outt$ so they are weakly separated from all sets in $\mathcal{D}^\inn_{\S}$ (see \cite{DKK14}).

  Now we want to do the following: we consider the part $G^\outt$ of $G$ outside $O$ and the part $G^\inn$ inside $O$. We would like to think of $G^\inn$ as a separate plabic graph (note that it has exactly $r$ vertices on the boundary -- the midpoints of $e_0,\dots,e_{r-1}$). We add $n-r$ boundary leaves to the boundary of $G^\inn$ so that the midpoint of each $e_i$ would be labeled by $q_i$ (recall that $\act(\S)=\{q_1<q_2<\ldots<q_r\}$). We make each of these boundary leaves black or white according to the color function $\col_\S$ that we constructed earlier.

  Let us make the following claims:
  \begin{enumerate}
   \item If $G$ is reduced then $G^\inn$ is reduced as well;
   \item The decorated permutation of $G^\inn$ is $(\sp,\col_\S)$;
   \item Let $G_1=G^\outt\cup G^\inn_1$ be the same as $G$ with $G^\inn$ replaced by another reduced plabic graph $G^\inn_1$ with decorated permutation $\sp$. Then $G_1$ is reduced and has the same decorated permutation as $G$.
  \end{enumerate}

  The first claim is clear from the definition of a reduced plabic graph (Definition~\ref{dfn:reduced}). The second claim is also easy to show: if $i_t=\sigma(q_t)=\pi(q_s)=j_s$ and if $(\gamma,\col_\S)$ is the decorated permutation of $G^\inn$ then $\gamma(q_s)=q_t$. But $q_t=\sp(s)$ which proves the second claim. Finally, the third claim follows from the results of \cite{P} that any two reduced plabic graphs with the same decorated permutation are connected by a sequence of square moves: $G^\inn$ and $G^\inn_1$ are connected by a sequence of square moves and each of them changes neither the reducedness of $G$ nor its decorated permutation.

  The three claims above show that reduced plabic graphs with decorated permutation equal to $\sp$ are literally the same as reduced plabic graphs that can occur inside $\S$. It remains to note that the strand labeled by $i$ in $G$ is labeled by $\sinv(i)$ in $G^\inn$ when we ignore $G^\outt$. Thus $\sinv$ is an isomorphism between $\S$ and $\CC(\sp)$.

 \end{proof}

\newcounter{eqn}
\renewcommand*{\theeqn}{\alph{eqn})}
\newcommand{\num}{\refstepcounter{eqn}\text{\theeqn}\quad}
\def\rt{{\widetilde{r}}}
\def\at{{\widetilde{ \alpha}}}
\def\bt{{\widetilde{ \beta}}}
\def\gt{{\widetilde{ \gamma}}}
\def\dt{{\widetilde{ \delta}}}

\section{Description of the elements in $\A_{I,J}$}\label{sec:description}
We now proceed to the last steps needed to prove Theorems \ref{thm:main} and \ref{thm:main_nc}. We are going to prove only the stronger Theorem~\ref{thm:main_nc} since we get almost no extra complications.

\def\proj{ \operatorname{proj}}
Throughout Sections~\ref{sec:description} and~\ref{sec:Chord_separation}, for fixed $I$ and $J$, if $R\in {[n]\choose m}$ then we denote $\proj(R)\in{[2k]\choose k}$ its image after ignoring all the elements from $\overline{I\bigtriangleup J}$. If there is no confusion, we also denote $\proj(R)$ by $\Rt$. For a number $r \in [n]$, if $r \in I\bigtriangleup J$ then we denote by $\rt$ the unique element in the set $\proj(r)$. For three sets $A,B,C\subset [n]$, we say that \emph{$A\subset B$ on $C$} if $A\cap C\subset B\cap C$.

Recall that $\A_{I,J}$ is the collection of subsets in ${[n]\choose m}$ that are weakly separated from both $I$ and $J$. Also recall that $\It$ and $\Jt$ partition the circle $[2k]$ into intervals $(\Pt_1,\Pt_2,\dots,\Pt_{2u})$ of lengths $(p_1,p_2,\dots,p_{2u})$ where $\Pt_{2i+1}\subset \It$ and $\Pt_{2i}\subset \Jt$. Finally, recall that $I,J$ form a \emph{balanced pair} iff $p_i+p_j<k$ for all $i\neq j\in [2u]$. The following lemma describes the elements of $\A_{I,J}$ for a balanced pair $I,J$.


\begin{lemma}\label{lemma:characterizationI}
Let $I,J \in {[n]\choose m}$ form a balanced pair. Let $R\in \A_{I,J}$ be weakly separated from both $I$ and $J$. Then there exists a four-tuple of cyclically ordered elements $\alpha<\beta\leq \gamma <\delta\leq \alpha$ of $[n]$ such that
\begin{alignat*}{6}
  \num& I\subset R \subset J &\text{ or }& J\subset R\subset I  &  \text{ on }(\alpha,\beta); \\
\num&                        &           & R \subset (I\cap J)  &  \text{ on }[ \beta, \gamma]; \\
  \num& I\subset R \subset J &\text{ or }& J\subset R\subset I  &  \text{ on }(\gamma, \delta); \\
\num&                        &           & (I\cup J) \subset R  &  \text{ on }[ \delta, \alpha].
\end{alignat*}
Moreover, each of $\proj[ \alpha, \beta)$ and $\proj( \gamma, \delta]$ is contained in a single interval associated with $\It$ and $\Jt$: $\proj( \gamma, \delta]\subset \Pt_l$ and $\proj[ \alpha, \beta)\subset\Pt_r$ for some $r,l\in [2u]$.
\end{lemma}

\begin{remark}
If $I$ and $J$ do not form a balanced pair, $\A_{I,J}$ contains more complicated elements. This is the main motivation for introducing this notion.
\end{remark}

We call $\Pt_l$ and $\Pt_r$ the (resp., left and right) \textit{endpoints} of $R$. The collection $\Pt_{l+1},\dots,\Pt_{r-1}$ of all other intervals that are contained in $R$ is called the \textit{internal part} of $R$: we have $(\Pt_{l+1}\cup\dots\cup\Pt_{r-1})\subset\proj R$.

\begin{proof}
Since $I,J,R$ are all of the same size, we can use Remark~\ref{rmk:ws_definition} as a definition of weak separation. Let $c_I$ (resp., $c_J$) be the (combinatorial) chords separating $I\setminus R$ from $R\setminus I$ (resp., $J\setminus R$ from $R\setminus J$). We view $c_I$ and $c_J$ as directed arrows so that $I\setminus R$ is to the left of $c_I$ and $R\setminus I$ is to the right of $c_I$. There are five possible ways (see Figure~\ref{fig:arrows}) to position $c_I$ and $c_J$ relative to each other.

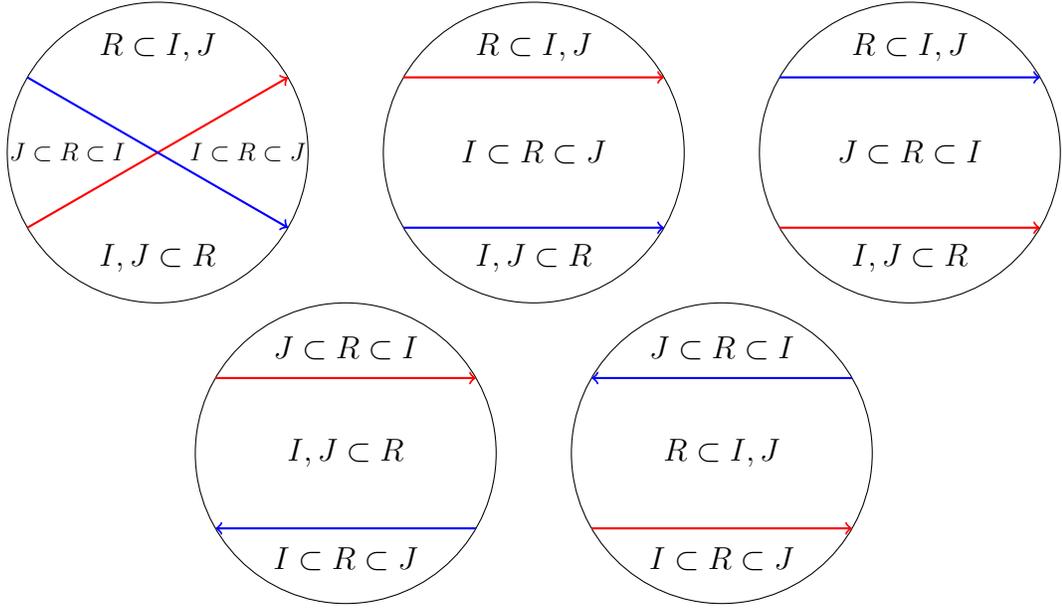
\begin{figure}
 \begin{tikzpicture}

\newcommand*{\rad}{2}%
\newcommand*{\stepH}{5}%
\newcommand*{\stepU}{4}%
\newcommand*{\angl}{30}%
\newcommand{\arrow}[5]{ 
    \draw[#3,->,thick] ({#1*\stepH -#4*\rad *cos(\angl)},{#2*\stepU -#5*\rad *sin(\angl)}) -- ({#1*\stepH +#4*\rad *cos(\angl)},{#2*\stepU +#5*\rad *sin(\angl)});
  }%
  \newcommand{\arrowH}[5]{ 
    \draw[#3,->,thick] ({#1*\stepH -#4*\rad *cos(\angl)},{#2*\stepU +#5*\rad *sin(\angl)}) -- ({#1*\stepH +#4*\rad *cos(\angl)},{#2*\stepU +#5*\rad *sin(\angl)});
  }%

\newcommand*{\vertM}{0.7}%
\newcommand*{\horM}{0.6}%
\newcommand*{\ssc}{0.8}%

\newcommand{\textR}[7]{%
  \draw ({#1*\stepH +#3*\rad *\horM +#6},{#2*\stepU +#4*\rad *\vertM +#7}) node {#5}
  }%

  \draw ({0*\stepH},\stepU) circle (\rad);
  \draw ({1*\stepH},\stepU) circle (\rad);
  \draw ({2*\stepH},\stepU) circle (\rad);
  \draw ({0.5*\stepH},0) circle (\rad);
  \draw ({1.5*\stepH},0) circle (\rad);

  \arrow{0}{1}{red}{+1}{+1};
  \arrow{0}{1}{blue}{+1}{-1};
  \textR{0}{1}{0}{1}{$R\subset I,J$}{0}{0};
  \textR{0}{1}{0}{-1}{$I,J\subset R$}{0}{0};
  \textR{0}{1}{1}{0}{\scalebox{\ssc}{$I\subset R\subset J$}}{0}{0};
  \textR{0}{1}{-1}{0}{\scalebox{\ssc}{$J\subset R\subset I$}}{0}{0};

  \arrowH{1}{1}{red}{+1}{+1};
  \arrowH{1}{1}{blue}{+1}{-1};
  \textR{1}{1}{0}{1}{$R\subset I,J$}{0}{0};
  \textR{1}{1}{0}{-1}{$I,J\subset R$}{0}{0};
  \textR{1}{1}{0}{0}{$I\subset R\subset J$}{0}{0};

  \arrowH{2}{1}{blue}{+1}{+1};
  \arrowH{2}{1}{red}{+1}{-1};
  \textR{2}{1}{0}{1}{$R\subset I,J$}{0}{0};
  \textR{2}{1}{0}{-1}{$I,J\subset R$}{0}{0};
  \textR{2}{1}{0}{0}{$J\subset R\subset I$}{0}{0};

  \arrowH{0.5}{0}{red}{+1}{+1};
  \arrowH{0.5}{0}{blue}{-1}{-1};
  \textR{0.5}{0}{0}{1}{$J\subset R\subset I$}{0}{0};
  \textR{0.5}{0}{0}{-1}{$I\subset R\subset J$}{0}{0};
  \textR{0.5}{0}{0}{0}{$I,J\subset R$}{0}{0};

  \arrowH{1.5}{0}{blue}{-1}{+1};
  \arrowH{1.5}{0}{red}{+1}{-1};
  \textR{1.5}{0}{0}{1}{$J\subset R\subset I$}{0}{0};
  \textR{1.5}{0}{0}{-1}{$I\subset R\subset J$}{0}{0};
  \textR{1.5}{0}{0}{0}{$R\subset I,J$}{0}{0};

 \end{tikzpicture}
 \caption{\label{fig:arrows} Five ways the relative positions of $c_I$ (red) and $c_J$ (blue) can look like. We have $R\subset I$ to the left of $c_I$ and $I\subset R$ to the right of $c_I$ and similarly for $c_J$.}
\end{figure}

Note that for the three circles on the top of Figure~\ref{fig:arrows}, we have nothing to prove: properties (1)-(4) clearly hold with numbers $\alpha, \beta, \gamma, \delta$ chosen to lie at the endpoints of the arrows; and since we have either $I\subset J$ or $J\subset I$ on both $( \alpha, \beta)$ and $( \gamma, \delta)$, each of these intervals has to map under $\proj$ to a subset of an interval $\Pt_i$ for some $i$.

The only thing left to show is that the two cases on the bottom of Figure~\ref{fig:arrows} lead to a contradiction. Consider the bottom-left circle. In this case $(I\cap J)\subset R$, therefore $|\Rt|\leq k$. Moreover, we have $\It,\Jt\subset \Rt$ between the arrows. But note that since $\It$ and $\Jt$ are complementary, the part of the circle above the red arrow has to lie inside a single interval $\Pt_i$ for some $i$, because for this part we have $J\subset R\subset I$ and thus $\Jt= \emptyset$ above the red arrow. Similarly, $\It= \emptyset$ below the blue arrow, thus the part of the circle below the blue arrow is contained inside a single interval $\Pt_j$ for some $j$. Since $I,J$ form a balanced pair, $p_i+p_j<k$ and since on the rest of the intervals $\Rt=\It\cup\Jt$, it follows that $|\Rt|>k$, a contradiction. The bottom-right case is treated in an analogous way with all the inequalities reversed, e.g. we first note that $R\subset (I\cup J)$ and in the end get a contradiction with the fact that $|\Rt|<k$ since it is only contained in at most two intervals.
\end{proof}

\section{Chord separation}\label{sec:Chord_separation}

In this section, we prove a few technical results on plabic graphs that we will use later to prove Theorem~\ref{thm:main_nc}. Their proofs are written in the language of \emph{plabic tilings} of~\cite{P}, which are the objects dual to reduced plabic graphs. We refer the reader to either of~\cite{P,G} for the background.

\def\taunk{{\tau_{k,n}}}

\begin{lemma}\label{lemma:strands1n_essential}
 Let $G$ be a reduced plabic graph with decorated permutation $\taunk$, for $1\leq k\leq n-1$, and suppose that $e$ is an edge such that the strands labeled $1$ and $n$ traverse $e$ in the opposite directions. Then either $e$ is a boundary edge or its endpoints are of different colors.
\end{lemma}
\begin{proof}
 Consider the plabic tiling dual to $G$. Then the edge $e^*$ dual to $e$ connects two sets $T\cup 1 $ and $T\cup n$ for some $T\in{[2,n-1]\choose k-1}$. We claim that $e^*$ is either a boundary edge or belongs to the intersection of a black and a white clique. Suppose that $e^*$ is not a boundary edge, thus we may assume that it belongs to some clique $C$. If $C$ is white then its boundary vertices listed in cyclic order are $K\cup x_1,K\cup x_2,\dots,K\cup x_r$ for some $x_1<x_2<\dots<x_r$ and $K\in {[n]\choose k-1}$. If $C$ is black then its boundary vertices listed in cyclic order are $L\setminus x_1,L\setminus x_2,\dots,L\setminus x_r$ for some $x_1<x_2<\dots<x_r$ and $L\in {[n]\choose k+1}$. In either case, $1$ and $n$ have to belong to $\{x_1,x_2\dots,x_r\}$ and the only way for this to happen is if $x_1=1$ and $x_r=n$. Therefore $e^*$ belongs to the boundary of the corresponding clique $C$. Since $e^*$ is not a boundary edge itself, there must be some other clique $C'$ with a boundary edge $e^*$, and $C'$ must be of a different color than $C$. This finishes the proof of Lemma~\ref{lemma:strands1n_essential}.
\end{proof}

\begin{lemma}\label{lemma:strands1n_unique_edge}
 Let $G$ be a reduced plabic graph with decorated permutation $\taunk$, for $1\leq k\leq n-1$. Then there exists a unique edge $e$ such that the strands labeled $1$ and $n$ traverse $e$ in the opposite directions.
\end{lemma}
\begin{proof}
It is clear that such an edge $e$ exists because the strands $1$ and $n$ have to intersect when $1<k<n-1$, and for $k=1$ or $k=n-1$ there is a boundary edge that satisfies the requirements of the lemma. So we only need to prove uniqueness. Suppose $e_1$ and $e_2$ are such edges, and let $T\cup 1,T\cup n$ be the labels of faces adjacent to $e_1$ and $S\cup 1,S\cup n$ be the labels of faces adjacent to $e_2$. Then the four sets $T\cup 1,T\cup n,S\cup 1,S\cup n$ must be weakly separated. Since $S$ and $T$ have the same size and $S\neq T$, it follows that there exist integers $s,t$ such that $s\in S\setminus T$ and $t\in T\setminus S$. If $s<t$ then $T\cup 1$ is not weakly separated from $S\cup n$ and if $t<s$ then $T\cup n$ is not weakly separated from $1\cup S$. We get a contradiction.

\end{proof}

\def\WS{\W}
We now review a result of the second author~\cite{G}.
\begin{dfn}[\cite{G}]
 We say that $S$ and $T$ are \emph{chord separated} if either $S$ surrounds $T$ or $T$ surrounds $S$. Equivalently, they are chord separated if there exist no cyclically ordered integers $a,b,c,d\in [n]$ satisfying $a,c\in S\setminus T$ and $b,d\in T\setminus S$.
\end{dfn}
As we have already noted in Remark~\ref{rmk:ws_definition}, it follows that if $|S|=|T|$ then they are chord separated if and only if they are weakly separated. In general, $S$ and $T$ may be chord separated but not weakly separated, for example, if $S=\{1,3\}$ and $T=\{2\}$. It turns out that the purity phenomenon also occurs for chord separation:
\begin{thm}[\cite{G}]\label{thm:chord}
 Every maximal by inclusion chord separated collection $\W\subset 2^{[n]}$ is also maximal by size:
 \[|\WS|={n\choose 0}+{n\choose 1}+{n\choose 2}+{n\choose 3}.\]
\end{thm}

It is also shown in~\cite{G} that maximal chord separated collections correspond to \emph{admissible families of reduced trivalent plabic graphs}. That allows us to prove the following generalization of Theorem~\ref{thm:LR}, part~(\ref{item:LR_sequence}):

\begin{lemma}\label{lemma:chord}
 Suppose that $\WS\subset2^{[n]}$ is a maximal by inclusion chord separated collection, and let $U\subset V\subset [2,n-1]$ be such that
 \[U,(1\cup U),(U\cup n),(1\cup U\cup n),V,(1\cup V),(V\cup n),(1\cup V\cup n)\in\WS.\]
 Let $u=|U|$ and $v=|V|$ be their sizes. Then there exists a sequence
 \[U=S_u\subset S_{u+1}\subset\dots\subset S_{v}=V\]
 of subsets of $[n]$ such that for all $k=u,\dots,v$, $|S_k|=k$ and
 \[S_k,(1\cup S_k),(S_k\cup n),(1\cup S_k\cup n)\in\WS.\]
\end{lemma}
\begin{proof}
\def\TPTiling{{\Sigma}}
\def\Vert{ \operatorname{Vert}}

 Consider the admissible family $\TPTiling_*=(\TPTiling_0,\TPTiling_1,\dots,\TPTiling_n)$ of triangulated plabic tilings that corresponds to $\WS$. The collection of vertices of the tiling $\TPTiling_k$ is denoted by $\Vert(\TPTiling_k)\subset {[n]\choose k}$ and satisfies
 \[\Vert(\TPTiling_k)=\WS\cap{[n]\choose k}.\]
  Let us explain the meaning of \emph{admissible}. Take any $0<k<n$ and any edge $e\in \TPTiling_k$ with endpoints $T,T'\subset [n]$ (in the dual trivalent plabic graph of $\TPTiling_k$, $T$ and $T'$ are labels of two faces sharing an edge). Then \emph{admissible} means that $T\cup T'$ and $T\cap T'$ must belong to $\WS$ as well, so $T\cup T'\in\Vert(\TPTiling_{k+1})$ and $T\cap T'\in\Vert(\TPTiling_{k-1})$.

We are now ready to define the sequence $U=S_u\subset\dots\subset S_v=V$. Note that $v\leq n-2$ because $V\subset [2,n-1]$. By Lemma~\ref{lemma:strands1n_unique_edge}, for each $u\leq k\leq v$ there is a unique edge $e_k$ in the plabic graph dual to $\TPTiling_{k+1}$ such that the strands labeled $1$ and $n$ traverse $e_k$ in the opposite directions. Thus the edge $e_k^*$ connects two sets $T\cup 1$ and $T\cup n$ for some $T\in {[n]\choose k}$. In this case, we set $S_k:=T$.

We need to show why $U=S_u$ and $V=S_v$ and why $S_{k-1}\subset S_{k}$ for each $u< k\leq v$. The first two claims are clear since the edge $e_k^*$ is unique but by Lemma~\ref{lemma:strands1n_essential}, there are edges in $\TPTiling_{u+1}$ and $\TPTiling_{v+1}$ that connect $U\cup 1$ with $U\cup n$ and $V\cup 1$ with $V\cup n$ respectively. Now let $u< k\leq v$ and we would like to see why $S_{k-1}\subset S_{k}$. By the definition of $S_{k-1}$, there is an edge $e_{k-1}$ in $\TPTiling_{k}$ that connects $S_{k-1}\cup 1$ with $S_{k-1}\cup n$. If $k=1$ then all the triangles in $\TPTiling_{k}$ are white, so $e_{k-1}$ belongs to the boundary of some white triangle in $\TPTiling_{k}$. Otherwise $1<k<n-1$ so $e_{k-1}$ belongs to the intersection of a white triangle with a black triangle in $\TPTiling_{k}$. In both cases, we get that $e_{k-1}$ is an edge of a white triangle in $\TPTiling_{k}$. We know that $S_{k-1}\cup 1$ and $S_{k-1}\cup n$ are vertices of this white triangle, thus there is an integer $a\in [n]$ such that $S_{k-1}\cup a$ is the third vertex of this white triangle. By the admissibility condition, the sets
\[(S_{k-1}\cup \{1,a\}),(S_{k-1}\cup \{a,n\}),(S_{k-1}\cup \{1,n\})\]
are the vertices of a black triangle in $\TPTiling_{k+1}$. In particular, there is an edge in $\TPTiling_{k+1}$ connecting $S_{k-1}\cup \{1,a\}$ with $S_{k-1}\cup \{a,n\}$. By Lemma~\ref{lemma:strands1n_unique_edge}, such an edge is unique and is therefore equal to $e_{k+1}$, which implies that $S_k=S_{k-1}\cup a$ and so $S_{k-1}\subset S_k$. We are done with the proof of Lemma~\ref{lemma:chord}.
\end{proof}

\section{Proof of Theorem~\ref{thm:main_nc}}\label{section:proof}
In this section, we are going to show that for any balanced pair $I,J\in{[n]\choose m}$, the collection $\A_{I,J}\subset {[n]\choose m}$ is a pure domain of rank
\[m(n-m)-k^2+2k+\sum_{i=1}^{2u} {p_i\choose 2}.\]

\subsection{Plan of the proof}

Let $\W$ be a maximal weakly separated collection inside $\A_{I,J}$. The sets $\It$ and $\Jt$ partition the circle $[2k]$ into $2u$ intervals $(\Pt_1,\Pt_2,\dots,\Pt_{2u})$ for some $u\geq 1$. Here $\Pt_i\in {[2k]\choose k}$ for all $i=1,\dots, 2u$. We set $p_i:=|\Pt_i|$.

Our proof is going to consist of several quite involved steps, so we start with the outline of these steps. The first one is
\begin{lemma}\label{lemma:LR}
 For each $1\leq i\leq 2u$, there exists a sequence of subsets
 \[\emptyset=\St_0^i\subset \St_1^i\subset\dots\subset \St_{p_i}^i=\Pt_i\]
such that $|\St_t^i|=t$ for $0\leq t\leq p_i$ and, moreover, any set in $\W$ is weakly separated from any $S\in \A_{I,J}$ that satisfies the following conditions:
  \begin{enumerate}[label=(\alph*)]
   \item\label{item:predictable} $\left(I\cap J\right)\subset S\subset \left(I\cup J\right)$;
   \item \label{item:endpoints} if $\Pt_l$ and $\Pt_r$ are the endpoints of $\St$ then $\St\cap \Pt_l=\St_t^l$ and $\St\cap \Pt_r=\St_q^r$ for some $0\leq t\leq p_i$ and $0\leq q\leq p_r$.
  \end{enumerate}
\end{lemma}

This lemma is going to be an application of Lemma~\ref{lemma:chord} together with Lemma~\ref{lemma:characterizationI}. Now consider all the sets $S\in \A_{I,J}$ satisfying conditions~\ref{item:predictable} and~\ref{item:endpoints}. It is easy to see that such sets form a generalized cyclic pattern $\S=(S_1,\dots,S_{2k})$ that we now define. First, note that the map $\proj$ is a bijection between the sets $S\in{[n]\choose m}$ that satisfy~\ref{item:predictable} and the sets $\St\in {[2k]\choose k}$. Consider a set $S_m$ satisfying~\ref{item:predictable}-\ref{item:endpoints}. We have
\[\proj(S_m)=\St_t^i\cup \Pt_{i+1}\cup\dots\cup \Pt_{j-1}\cup \St_u^j.\]
Then we can define ``the next'' set $S_{m+1}$ by defining its projection:
\[\proj(S_{m+1})=\begin{cases}
           \St_{t-1}^i\cup \Pt_{i+1}\cup\dots\cup \Pt_{j-1}\cup \St_{u+1}^j, & \text{if } u<p_j;\\
           \St_{t-1}^i\cup \Pt_{i+1}\cup\dots\cup \Pt_{j-1}\cup \Pt_j\cup \St_1^{j+1}, & \text{if } u=p_j.
          \end{cases}\]
Clearly, every subset $S\in \A_{I,J}$ satisfying conditions~\ref{item:predictable} and \ref{item:endpoints} appears somewhere in this sequence $\S=(S_1,S_2,\dots,S_{2k})$. Here $S_1$ is chosen so that its left endpoint is $\Pt_1$ and $\St_1\cap \Pt_1=\Pt_1$. It is clear that $\S$ is a generalized cyclic pattern.

Since $\D_\S^\inn$ and $\D_\S^\outt$ form a \emph{complementary pair} (see~\cite{DKK14}), and because the rank of $[n]\choose m$ equals $m(n-m)+1$, it follows that
\begin{equation}\label{eqn:ranks}
\rk \D_\S^\inn+\rk\D_\S^\outt=m(n-m)+1+|\S|=m(n-m)+2k+1.
\end{equation}

To calculate the ranks of $\D_\S^\inn$ and $\D_\S^\outt$, we use Theorem~\ref{thm:iso} to show that
\begin{lemma}\label{lemma:iso}
The generalized cyclic pattern $\proj\S$ is isomorphic (in the sense of Definition~\ref{dfn:iso}) to $\I(\tau_\It\circ\tau_{k,2k})$, the Grassmann necklace corresponding to the canonical decorated permutation $\tau_\It\circ\tau_{k,2k}$ associated with $\It$.
\end{lemma}
\def\Ical{ \mathcal{I}}
The Grassmann necklace $\I(\tau_\It\circ\tau_{k,2k})$ has a unique preimage $\hat \Ical$ under $\proj$ that consists of sets satisfying~\ref{item:predictable}. It is then clear that $\hat\Ical$ and $\S$ are isomorphic as well and that $\ell(\hat\Ical)=\ell(\tau_\It\circ\tau_{k,2k})$. By Lemma~\ref{lemma:iso} and Remark~\ref{rmk:length}, it follows that
\[\rk\D_\S^\inn=1+\ell(\tau_\It\circ\tau_{k,2k})=1+k^2-\sum_{i=1}^{2u} {p_i \choose 2}.\]

Substituting this into~(\ref{eqn:ranks}) yields
\[\rk\D_\S^\outt=m(n-m)-k^2+2k+\sum_{i=1}^{2u} {p_i \choose 2}.\]

Finally, using the geometric definition of domains $\D_\S^\inn$ and $\D_\S^\outt$ given in Proposition \ref{prop:geometricdef} we will show:

\begin{lemma}\label{lemma:outside}
 \[\W\subset \D_\S^\outt\subset \A_{I,J}.\]
\end{lemma}
This lemma (whose proof involves Theorem~\ref{thm:purity_of_generalizedcyclic}) implies that $\W$ is a maximal weakly separated collection inside $\D_\S^\outt$, and thus $|\W|=\rk\D_\S^\outt$,
which completes the proof of Theorem~\ref{thm:main_nc}. \qed

We now proceed to the proofs of Lemmas~\ref{lemma:LR},~\ref{lemma:iso} and~\ref{lemma:outside}.

\subsection*{Proof of Lemma~\ref{lemma:LR}}

\newcommand\Phat[1]{\widehat{P_#1}}
\newcommand\Qhat[1]{\widehat{Q_#1}}
\def\Rhat{\widehat{R}}
\def\phiI{{\phi_i(\Qhat{i})}}
\def\phiJ{{\phi_i(\Rhat{i})}}
\def\ehat{\widehat{e}}
\def\fhat{\widehat{f}}
\def\ghat{\widehat{g}}
\def\hhat{\widehat{h}}
\def\phat{\widehat{p}}
\def\ahat{\widehat{a}}
\def\bhat{\widehat{b}}
\def\chat{\widehat{c}}
\def\dhat{\widehat{d}}

\def\B{{\mathcal{B}}}

For any $1 \leq i \leq 2u$, let $e_i,h_i\in [2k]$ be such that $P_i=[e_i,h_i]$ as a circular interval. Let $\ehat_i,\hhat_i\in [n]$ be such that $\proj(\ehat_i)=e_i-1$ and $\proj(\hhat_i)=h_i+1$. Then define $\Phat{i}:=[\ehat_i,\hhat_i]\subset [n]$. Recall that $p_i=|P_i|$ so define $\phat_i=|\Phat{i}|$. Thus we have
\[\proj(\Phat{i})=\proj([\ehat_i,\hhat_i])=[e_i-1,h_i+1]\supset P_i=[e_i,h_i].\]

Let $\phi_i:2^{[n]} \rightarrow 2^{[\phat_i]}$ be the ``intersection+shift'' map defined by \[\phi_i(Q)=(Q\cap \Phat{i})-\ehat_i+1=\{q-\ehat_i+1\mid q\in Q\cap\Phat{i}\}\]
for every $Q\subset[n]$.

Consider the collection $\phi_i(\W)\subset 2^{[\phat_i]}$. It is clear that this is a chord separated collection since the relation of being chord separated is preserved with respect to shifting and omitting elements. Also note that $\ehat_i$ and $\hhat_i$ belong to the symmetric difference of $I$ and $J$, so $1,\phat_i$ belong to $\phi_i(I\cup J)$ but not to $\phi_i(I\cap J)$. Let $U=\phi_i(I\cap J)$ and $V=\phi_i(I\cup J)\setminus \{1,\phat_i\}$. Our goal is to apply Lemma~\ref{lemma:chord} to $\phi_i(\W)$, $U$ and $V$. In order to do that we need to show the following

\begin{lemma}\label{lemma:subsets_8}
 The collection $\phi_i(\W)$ is chord separated from the eight sets
\begin{equation}\label{eq:subsets_8}
 U,(1\cup U),(U\cup \phat_i),(1\cup U\cup\phat_i),V,(1\cup V),(V\cup \phat_i),(1\cup V\cup \phat_i).
\end{equation}
\end{lemma}

\begin{proof}

Since $P_i$ is an interval associated with $\It$ and $\Jt$, it follows that either $\It\cap P_i=\emptyset,\Jt\cap P_i=P_i$ or vice versa. These two cases are symmetric so let us assume that $\It\cap P_i=\emptyset$ and $\Jt\cap P_i=P_i$. But then we get $e_i-1,h_i+1\in\It$. Lifting all this from $[2k]$ to $[n]$ yields that $\phi_i(I)=1\cup U\cup \phat_i$ and $\phi_i(J)=V$. We need to see why the rest six subsets from~(\ref{eq:subsets_8}) are chord separated from $\W$.

Suppose for example that $1\cup U$ is not chord separated from some $\Rhat\in\phi_i(\W)$. Since $1\cup U\cup \phat_i$ is chord separated from $\Rhat$, there must be integers $\dhat<\fhat<\ghat\in [1,\phat_i-1]$ such that
\[\dhat,\ghat\in (1\cup U)\setminus \Rhat,\quad \fhat,\phat_i\in \Rhat\setminus (1\cup U).\]
Let $d,f,g\in \Phat{i}$ be their preimages under $\phi_i$, and let $\Rhat$ be the image of some $R\in\W$. We get that $d,f,g,h_i\in[n]$ are cyclically ordered so that
\begin{equation}\label{eq:gefh}
d\in I\setminus R;\quad f\in R\setminus I;\quad g\in (I\cap J)\setminus R;\quad h_i\in R\setminus J.
\end{equation}

Let $\alpha,\beta,\gamma,\delta\in[n]$ be the integers for $R$ from Lemma~\ref{lemma:characterizationI}. Since $R\subset (I\cap J)$ on $[\beta,\gamma]$, we get that $g\in [\beta,\gamma]$ but $f,h_i\not\in[\beta,\gamma]$, and hence $d\not\in[\beta,\gamma]$.
This leaves only two possibilities for $d$, namely, $d\in (\alpha,\beta)$ or $d\in(\gamma,\delta)$, because it cannot belong to $[\delta,\alpha]$ as $d\not\in R$. We therefore have either $J\subset R\subset I$ or $I\subset R\subset J$ on $(\alpha,\beta)$. We know that $\dhat\in I,\dhat\not\in R$ which means that we have $J\subset R\subset I$ on $(\alpha,\beta)$ and in particular $d\not\in J$ so $\dhat=1$ and $d=e_i$. If $d\in (\alpha,\beta)$ then the same is true for $f$. We get a contradiction since $f\in J\setminus I$ while we have $J\subset I$ on $(\alpha,\beta)$. Thus we get that $d\in (\gamma,\delta)$. This is impossible because then $\proj(\gamma,\delta)$ would intersect all the intervals associated with $\It$ and $\Jt$ which contradicts Lemma~\ref{lemma:characterizationI}. We have shown that $\phi_i(\W)$ is weakly separated from $1\cup U$. The proof that $U\cup \phat_i$ is weakly separated from $\phi_i(\W)$ is absolutely similar. Suppose now that $U$ is not weakly separated from some $\Rhat\in\phi_i(\W)$. Then there exist $a<b<c<d\in[\phat_i]$ such that either $a,c\in U\setminus \Rhat$ and $b,d\in \Rhat\setminus U$ or vice versa. In the first case, we get that $1\cup U$ is not weakly separated from $\Rhat$, and in the second case we get that $U\cup \phat_i$ is not weakly separated from $\Rhat$. In either case we get a contradiction with what we have previously shown and thus the four sets $U,1\cup U,U\cup \phat_i,1\cup U\cup\phat_i$ are weakly separated from $\phi_i(\W)$.

The proof for the remaining four sets from~(\ref{eq:subsets_8}) follows from the fact that the relation of chord separation is preserved under replacing all subsets in $\W$ by their complements. We are done with the proof of Lemma~\ref{lemma:subsets_8}.

\end{proof}

\newcommand\Shat[1]{\widehat{S_{#1}^i}}
Lemmas~\ref{lemma:chord} and~\ref{lemma:subsets_8} together with Theorem~\ref{thm:chord} provide a family of subsets
\[U=\Shat{u}\subset\Shat{u+1}\dots\Shat{v}=V\]
of $[2,\phat_i-1]$ that are all chord separated from $\phi_i(\W)$. Here $u=|U|$ and $v=|V|$ just as in Lemma~\ref{lemma:chord}. For each $t=0,1,\dots,p_i=v-u$, we let $\St_t^i$ be the unique preimage of $\Shat{u+t}$ under $\phi_i$ that is contained in $\Phat{i}$.

Instead of showing Lemma~\ref{lemma:LR}, we will show the following stronger statement:

\begin{lemma}\label{lemma:a_ib_i}
Let $R\in \W$ be a set and consider the numbers $ \alpha< \beta\leq \gamma< \delta\leq \alpha$ given by Lemma~\ref{lemma:characterizationI}. Then for each $1\leq i\leq 2k$, there exist two numbers $b_i\in [ \alpha, \beta)$ and $a_i\in (\gamma, \delta]$ such that the (combinatorial) chord through $a_i$ and $b_i$ separates $R\setminus S_{i}$ from $S_i\setminus R$ and the same chord separates $R\setminus S_{i-1}$ from $S_{i-1}\setminus R$. More specifically, $a_i$ and $b_i$ satisfy
\begin{enumerate}
 \item \label{item1}$R\setminus S_{i}\subset [a_i,b_i]$;
 \item $R\setminus S_{i-1}\subset [a_i,b_i]$;
 \item $S_i\setminus R\subset (b_i,a_i);$
 \item \label{item4}$S_{i-1}\setminus R\subset (b_i,a_i).$
\end{enumerate}
\end{lemma}
\begin{proof}
Note that we do not know yet that $R$ and $S_i$ are weakly separated, as this is the result of Lemma~\ref{lemma:LR}. Of course, if we separate $R\setminus S_i$ from $S_i\setminus R$ by a chord, we can immediately deduce Lemma~\ref{lemma:LR} from this because their sizes coincide.

Let $R\in \A_{I,J}$. We define $a_i$ and $b_i$ as follows: the interval $[a_i,b_i]$ is any maximal (by inclusion) interval satisfying
\begin{equation}\label{eq:aibi}
\left(S_i\cup S_{i-1}\right)\subset R\quad \text{ on $[a_i,b_i]$}
\end{equation}
subject to the conditions
\[b_i\in [ \alpha, \beta),\quad a_i\in (\gamma, \delta].\]

Note that there exists at least one interval $[a_i,b_i]$ satisfying (\ref{eq:aibi}), namely, $[\delta, \alpha]$ on which by Lemma~\ref{lemma:characterizationI} we have $I\cup J\subset R$ (recall that $S_i$ and $S_{i-1}$ satisfy condition~\ref{item:predictable} by definition).

We are going to show that for thus defined $a_i,b_i$ we also get
\[R\subset (S_i\cap S_{i-1})\quad \text{ on $(b_i,a_i)$}.\]

Suppose this is false:
there is some element $c\in R\cap (b_i,a_i)$ such that $c\not\in S_i\cap S_{i-1}$. There are two symmetric options: $c\in [\alpha, \beta)$ and $c\in ( \gamma, \delta]$. Without loss of generality assume $c\in [\alpha, \beta)$. Then by maximality of $[a_i,b_i]$ there exists an element $d$ satisfying $b_i<d<c$ such that $d\in S_i\cup S_{i-1}$ but $d\not\in R$.

By Lemma~\ref{lemma:characterizationI}, the set $\proj[\alpha,\beta)$ is contained in $P_r\subset [2k]$ which is the right endpoint of $R$. Recall the notation $\Phat{r}=[\ehat_r,\hhat_r]$ from the beginning of this section. We have that $b_i<d<c$ all belong to $\Phat{r}$ so let us apply $\phi_r$ to them to get elements $\bhat_i<\dhat<\chat$ in $[\phat_r]$. We know that $c\in [\alpha,\beta)$ and $\proj[\alpha,\beta)\subset P_r$ which actually implies that $1<\bhat_i<\dhat<\chat<\phat_r$. Let $\Rhat:=\phi_r(R)$. Now, we know that $d\in S_i\cup S_{i-1}$ but $d\not\in R$. We will assume that $d\in S_i$ for simplicity, the case $d\in S_{i-1}$ is completely analogous.

\newcommand\Shatr[1]{\widehat{S_#1^r}}
By the definition of $S_i$, there is an index $t\in [0,p_r]$ such that
\[(\phi_r(S_i)\cap [2,\phat_r-1])=\Shatr{t}.\]
Hence we have $d\in \Shatr{t}\setminus \Rhat$ and $c\in \Rhat\setminus \Shatr{t}$, and we know that $\Rhat$ is chord separated from
\[\Shatr{t},1\cup \Shatr{t},\Shatr{t}\cup\phat_r,1\cup\Shatr{t}\cup\phat_r.\]
Therefore we have either $1\not\in\Rhat$ or $\phat_r\in\Rhat$, because otherwise if $1\in\Rhat$ and $\phat_r\not\in\Rhat$, $\Rhat$ would not be chord separated from $\Shatr{t}\cup\phat_r$. So suppose $1\not\in\Rhat$. Then in order for $1\cup\Shatr{t}$ and $\Rhat$ to be chord separated, we must have $2,\dots,d-1\not\in\Rhat\setminus \Shat{r}$.
Applying the inverse of $\phi$ implies that the integers
\[\ehat_r<\alpha\leq d<c<\beta\leq \hhat_r\]
are cyclically ordered and
\[\ehat_r, d\not\in R,\quad c,\hhat_r\in R.\]
Moreover, on $[\ehat_r,d]$ we must have $R\subset I\cup J$ in order for $2,\dots,d-1\not\in\Rhat\setminus \Shat{r}$. Note that $\ehat_r\in I\cup J$ by definition. The fact that $\ehat_r\not\in\Rhat$ means that $\ehat_r\not\in [\delta,\alpha]$ where $I\cup J\subset R$. Therefore $[\delta,\alpha]\subset (\ehat_r,d]$. It follows that on $[\delta,\alpha]$ we have $R=I\cup J$. Thus the set $R$ is contained in $I\cup J$ (on $[n]$) and hence $\proj R$ contains at least $k$ elements. Since $\delta\in (\ehat_r,d]$, the image $\proj(\gamma,\delta)$ belongs to either $P_r$ or $P_{r-1}$. The conclusion is, the images of $(\gamma,\delta),[\delta,\alpha],$ and $(\alpha,\beta)$ under $\proj$ are all contained in two intervals $P_r,P_{r-1}$ associated with $\It$ and $\Jt$. But on the remaining interval $[\beta,\gamma]$ we have $R\subset I\cap J$ so $\proj R$ contains no elements on the complement of $P_r\cup P_{r-1}$. We get a contradiction since $\proj R$ has at least $k$ elements but is supported on just two intervals associated with $\It,\Jt$ which form a balanced pair. We are done with the case $1\not\in\Rhat$. The case $\phat_r\in\Rhat$ follows by replacing each set with its complement. We are done with Lemma~\ref{lemma:a_ib_i} and thus with Lemma~\ref{lemma:LR}.
\end{proof}

\subsection*{Proof of Lemma~\ref{lemma:iso}}

Define $0=q_0,q_1,\dots,q_{2u}=2k$ to be the partial sums of $p_i$'s:
\[q_i:=p_1+p_2+\ldots+p_i.\]
Then consider the permutation $\gamma=(\gamma(1),\gamma(2),\dots,\gamma(2k))$ defined in such a way that for every $1\leq i\leq 2u$ and for every $0\leq t\leq p_i$ we have
\[\St_t^i=\{\gamma(q_i+1),\gamma(q_i+2)\dots,\gamma(q_i+t)\}.\]
Note that our pattern $\proj\S$ clearly satisfies Definition~\ref{dfn:gr_like_necklace} since $\act(\proj\S)=[2k]$. Let $\sigma$ and $\pi$ be the permutations of $[n]$ that we assigned to $\proj\S$ in Section~\ref{sect:iso}. Then
\[\sigma=\left(\gamma(q_1),\gamma(q_1-1),\dots,\gamma(1),\gamma(q_2),\dots,\gamma(q_1+1),\dots,\gamma(q_{2u}),\dots,\gamma(q_{2u-1}+1)\right)\]
in one-line notation. In other words, $\sigma=\gamma\circ\tau_\It$. Similarly, since $\proj S_1=P_1\cup\dots\cup P_i\cup \St_t^i$ for some $i$ and $t$ and $|\proj S_1|=k$, we have $\proj S_1=\{\gamma(1),\dots,\gamma(k)\}$ and hence
\[\pi=\left(\gamma(k+1),\gamma(k+2),\dots,\gamma(2k),\gamma(1),\gamma(2),\dots,\gamma(k)\right)=\tau_{k,2k}.\]
Because $\tau_\It$ and $\tau_{k,2k}$ are involutions, it follows that
\[\sp=\tau_\It\circ\gamma^{-1}\circ\gamma\circ\tau_{k,2k}=\tau_\It\circ\tau_{k,2k},\]
and we are done by Theorem~\ref{thm:iso}. \qed

\def\RR{{\mathbb{R}}}
\def\SSS{{\proj(\S)}}
\subsection*{Proof of Lemma~\ref{lemma:outside}}
We need to prove $\W\subset\D_\S^\outt$ and $\D_\S^\outt\subset\A_{I,J}$. We start with the second inclusion. To show it, we use Theorem~\ref{thm:purity_of_generalizedcyclic} which states that every element of $\D_\S^\inn$ is weakly separated from every element of $\D_\S^\outt$. Therefore if we show that $I,J\in\D_\S^\inn$, then every element of $\D_\S^\outt$ will be weakly separated from $I$ and $J$ and thus belong to $\A_{I,J}$. But since every element $S$ of $\S$ satisfies condition~\ref{item:predictable}:
\[\left(I\cap J\right)\subset S\subset \left(I\cup J\right),\]
we get (using the geometric definition of $\D_\S^\inn$) that $I$ and $J$ belong to $\D_\S^\inn$ iff $\It$ and $\Jt$ belong to $\D_\SSS^\inn$ where $\SSS=(\St_1,\dots,\St_{2k})$ is the image of $\S$ under $\proj$. The fact that $\It$ and $\Jt$ belong to $\D_\SSS^\inn$ is easy to see using Remark~\ref{rmk:positroid}. This shows that $\D_\S^\outt\subset\A_{I,J}$.

Next, we want to show the first inclusion $\W\subset\D_\S^\outt$. In other words, for any $R\in \W$ we need to show that the point in $\RR^2$ corresponding to $R$ lies outside the simple closed curve $\zeta_\S$ corresponding to $\S$. Note that $\S$ has to satisfy some additional properties from Theorem~\ref{thm:purity_of_generalizedcyclic} in order for $\zeta_\S$ to be non self-intersecting. These properties are:
\begin{enumerate}
 \item $\S$ contains no quadruple $S_{p-1},S_p,S_{q-1},S_q$ such that $\{S_{p-1},S_p\}=\{Xi,Xk\}$ and $\{S_{q-1},S_q\}=\{Xj,Xl\}$, where $i<j<k<l$;
 \item $\S$ contains no quadruple $S_{p-1},S_p,S_{q-1},S_q$ such that $\{S_{p-1},S_p\}=\{X \setminus i,X \setminus k\}$ and $\{S_{q-1},S_q\}=\{X \setminus j,X \setminus l\}$, where $i<j<k<l$.
\end{enumerate}
It is clear that $\S$ satisfies both of them, because for each $X$ of size $k-1$ (resp., $k+1$), there are at most two elements $S_p,S_q$ in $\S$ that contain (resp., are contained in) $X$.

The question of whether $R$ lies inside or outside of $\zeta_\S$ is not immediately clear because $\zeta_\S$ is in general not a convex polygon, so it can happen that there is no line in $\RR^2$ separating $R$ from $\zeta_\S$. We now give a simple sufficient condition for a point in $\RR^2$ to be outside of a polygonal closed curve in $\RR^2$.

\begin{lemma}\label{lemma:geometry}
 Let $\<\cdot,\cdot\>$ be the usual inner product in $\RR^2$ and let $v\in\RR^2$ be a point. Consider a closed polygonal chain $L=(u_0,u_1,\dots,u_r=u_0)$. Assume that there are vectors $w_i,1\leq i\leq r$ such that $\<v,w_i\> \geq \<u_i,w_i\>$ and $\<v,w_i\> \geq \<u_{i-1},w_i\>$ for all $1\leq i\leq r$. Suppose in addition that all these vectors belong to the same half-plane, i.e. there is a single vector $w$ such that $\<w,w_i\>>0$ for all $1\leq i\leq r$. Then either $v\in L$ or $v$ is outside of the region surrounded by $L$.
\end{lemma}
\begin{proof}[Proof of Lemma~\ref{lemma:geometry}]
 Consider a ray $R:=\{v+tw\}_{t\geq 0}$ starting at $v$. It is sufficient to show that $R$ does not intersect $L$. Equivalently, $R$ does not intersect each edge $[u_{i-1},u_i]$ for $1\leq i\leq r$. Suppose that happened for some $i$. Then there exist real numbers $s\in [0,1],t\in [0,+\infty)$ such that
 \[su_{i-1}+(1-s)u_i= v+tw.\]
 After we take the inner product of both sides with $w_i$, we get
 \[s\<u_{i-1},w_i\>+(1-s)\<u_i,w_i\>=\<v,w_i\>+t\<w,w_i\>.\]
 On the other hand, by the properties of $w_i$
 \[s\<u_{i-1},w_i\>+(1-s)\<u_i,w_i\>\leq s\<v,w_i\> +(1-s)\<v,w_i\>=\<v,w_i\>\leq \<v,w_i\>+t\<w,w_i\>.\]
 This implies that all the inequalities are actually equalities meaning that $t=0$ so $v\in L$.
\end{proof}

We would like to use now Lemma~\ref{lemma:geometry} to prove Lemma~\ref{lemma:outside}. To do that, we are going to use the geometric definition of $\D^\inn_\S$. We need to construct vectors $w_i,1\leq i\leq 2k$, that satisfy the requirements of Lemma~\ref{lemma:geometry}. We do this with the help of Lemma~\ref{lemma:a_ib_i}. Namely, number the elements of $I\bigtriangleup J$ by $(t_1,t_2,\dots,t_{2k})$ in cyclic order. Consider a convex $n$-gon with vertices $\xi_1,\dots,\xi_n\in\RR^2$ such that the vectors $(\xi_{t_i})_{i=1}^{2k}$ would be the vertices of a \emph{regular} $2k$-gon.

The edges of the curve $\zeta_\S$ are exactly line segments connecting $S_{i-1}$ to $S_i$ for some $i$. For each such edge we have a vector $\xi_{a_i}-\xi_{b_i}$ from Lemma~\ref{lemma:a_ib_i}. Let $w_i$ be a unit vector that is orthogonal to $\xi_{a_i}-\xi_{b_i}$ so that the orientation of $(\xi_{a_i}-\xi_{b_i},w_i)$ would be positive. Let $w$ be the vector orthogonal to
\[\xi_ \alpha+\xi_ \beta-\xi_ \gamma-\xi_ \delta.\]
The fact that the inner product $\<w_i,w\>$ is always positive for all $i\in [2k]$ follows from the following observations:
\begin{itemize}
 \item $I,J$ form a balanced pair;
 \item $a_i\in (\gamma, \delta]$ and $b_i\in [\alpha, \beta)$;
 \item each of $\proj(\gamma, \delta]$ and $\proj[\alpha, \beta)$ is contained in a single interval associated with $\It$;
 \item the union $(\gamma, \delta]\cup [\alpha, \beta)$ contains less than $k$ elements of $\{t_1,t_2,\dots,t_{2k}\}$.
\end{itemize}
Indeed, the last claim together with the way we arranged the vertices of our $n$-gon ensures that the angle between $w$ and $w_i$ is strictly less than $\pi/2$.

Finally, since the numbers $a_i,b_i$ satisfy properties~(\ref{item1})-(\ref{item4}) in the statement of Lemma~\ref{lemma:a_ib_i}, we get that the vectors $w_i$ satisfy the requirements of Lemma~\ref{lemma:geometry}. This finishes the proof of Lemma~\ref{lemma:outside} and hence we are also done with the proof of Theorem~\ref{thm:main_nc} (which implies Theorem~\ref{thm:main}). \qed

\section{The unbalanced case}\label{section:unbalanced}

In the previous sections, we found the exact value of $d(\cdot,\cdot)$ for balanced pairs. In this section, we concentrate on the unbalanced case. As we mentioned in Remark~\ref{unbalancedremark}, the purity phenomenon does not necessarily hold. However, we can still provide a lower bound for the size of the maximal weakly separated collections inside the domain, and as a result, to obtain an upper bound on $d(\cdot,\cdot)$ for the unbalanced case. This bound is described in Theorem~\ref{thm:lowerboununbalanced}.

In addition, in Theorem~\ref{thm:twointervals}, we evaluate the exact value of $d(\cdot,\cdot)$ for a wide family of unbalanced pairs. Recall that for an even positive integer $t$, $\P_t$ is the collection of all sets $I$ for which  $I$ and $\II$ partition the circle into $t$ intervals. In this section we concentrate on the structure of $\P_4$. This case is particularly interesting, since $t=4$ is the minimal $t$ for which $I, \II \in \P_t$ are not weakly separated. While sets $I \in \P_4$ can never be balanced, we were still able to find the exact value of $d(I,\II)$ as well as the maximal possible cardinality of a weakly separated collection in $\A_{I,\II}$. Moreover, we also find the value of $D(I,\II)$, the \textit{mutation distance} between $I$ and $\II$ introduced in \cite{FP}. The following definition and problem were introduced in \cite{FP}.
\begin{dfn}
Let $I,J\in {[n]\choose k}$ be any two $k$-element subsets in $[n]$.
Define the {\it mutation distance\/} $D(I,J)$ to be the minimal number of mutations needed
to transform a maximal weakly separated collection $\C_1 \subset {[n]\choose k}$ that contains $I$ into a maximal weakly separated collection $\C_2 \in {[n]\choose k}$ that
contains $J$.
\end{dfn}
\begin{problem}
How to calculate the mutation distance between any  $I$ and $J$, and how to
find the shortest chain of mutations between maximal weakly separated collections containing these subsets?
\end{problem}
Clearly, $D(I,J)=0$ iff $I$ and $J$ are weakly separated. In addition, $d(I,J)\leq D(I,J)$ for any pair $I,J \in {[n]\choose k}$.
Note that in the definition above, we consider only mutations that are square moves (that were discussed in previous sections). There are more general types of mutations, which are beyond the scope of this paper.

\begin{thm}
\label{thm:lowerboununbalanced}
Let $A\in\ch$ and let $(p_1,p_2,\dots,p_{2u})$ be the partition of the circle associated with $A$.
Define $$a_i:=p_{2i-1}, b_i:=p_{2i} \textrm{ for } 1 \leq i \leq 2u,$$ and
let $$\chi_{i,j} = \begin{cases} 0 &\mbox{if } i \in \{j, j-1\} \mbox{ or if } a_i+b_j<k\\ 1 &\mbox{otherwise}\end{cases}.$$
Then a maximal (by size) weakly separated collection in  $\A_{A,\AA}$ is of size at least
\begin{equation}\label{eq:unbalanced}
2k+\sum_{i=1}^u {a_i \choose 2}+\sum_{i=1}^u {b_i \choose 2}+\sum_{1\leq i<j \leq s}(a_i+b_j-k+1)\chi_{i,j}.
\end{equation}
\end{thm}
Note that by subtracting (\ref{eq:unbalanced}) from $1+k^2$ we get an upper bound on $d(A,\AA)$.
\begin{thm}
\label{thm:twointervals}
Let $A\in\ch$ such that $(p_1,p_2,p_3,p_4)$ is the partition of the circle associated with $A$ (so $u=2$).
Then the following three statements hold:
\begin{enumerate}
  \item A maximal weakly separated collection in $\A_{A,\AA}$ is of size at most
$2k+\sum_{i=1}^{4} {p_i\choose 2}$, and this bound is tight.
  \item $d(A,\AA)=1+k^2-2k-\sum_{i=1}^{2r} {p_i\choose 2};$
  \item Assume without loss of generality that $p_1 \geq \max(p_2,p_3,p_4)$ (we can always achieve that by rotation and switching the roles of $A$ and $\AA$). Then $D(I,\II)=p_2p_3p_4-2{p_3+1 \choose 3}.$
\end{enumerate}
\end{thm}

In order to prove theorem~\ref{thm:twointervals}, we start by presenting a certain projection of weakly separated collections in $\A_{A,\AA}$ into $\mathbb{R}^3$. Given $n>0$ and 4 positive integers $x,y,z,w$ that satisfy $x+y+z+w=n$, we define the projection $\phi_{x,y,z,w}: 2^{[n]} \rightarrow \mathbb{R}^4$ such that
$$\phi_{x,y,z,w}(C)_i = \begin{cases} |C \cap [1,x]| &\mbox{ if } i=1 \\ |C \cap [x+1,x+y]| &\mbox{ if } i=2 \\
|C \cap [x+y+1,x+y+z]| &\mbox{ if } i=3 \\|C \cap [x+y+z+1,x+y+z+w]| &\mbox{ if } i=4 \\ \end{cases}.$$

Let $\C \subset {[n] \choose k}$ be a weakly separated collection. Then $\phi_{x,y,z,w}(\C)$ lies in the intersection of $\mathbb{R}^4$ with the hyperplane $x_1+x_2+x_3+x_4=k$, so $\phi_{x,y,z,w}(\C)$ is, in fact, lying in a space isomorphic to $\mathbb{R}^3$. For simplicity, from now on we denote $\phi_{x,y,z,w}$ by $\phi$, unless we would like to specify the values of $x,y,z,w$. In the next two lemmas we described certain important properties of $\phi$.
\begin{lemma}\label{lemma:abovepyramid}
Let $\C \subset {[n] \choose k}$ be a weakly separated collection and let $J \in \C$. Consider the infinite square pyramid $P$ whose apex is $\phi(J)$ and whose $4$ edge vectors are $$\alpha_1=(0,0,-1,1), \alpha_2=(0,1,-1,0), \alpha_3=(-1,1,0,0), \alpha_4=(-1,0,0,1),$$
(that is, the pyramid consists of the elements $\phi(J)+\sum_{i=1}^4 t_i\alpha_i$ such that $t_i \geq 0$ for all $i$). Then for all $I \in \C$, $\phi(I)$ cannot lie in the interior of $P$. Similarly, $\phi(I)$ cannot lie in the interior of $-P$ (which is defined as $\phi(J)+\sum_{i=1}^4 t_i(-\alpha_i)$).
\end{lemma}
\begin{proof}
Assume in contradiction that there is an element $I \in \C$ such that $\phi(I)$ lies in the interior of $P$. Therefore $\phi(I)=\phi(J)+\sum_{i=1}^4 t_i\alpha_i$ such that either $t_1,t_3>0$ or $t_2,t_4>0$ (or both). Indeed, otherwise if $t_i=t_{i+1}=0$ for some $i=1,2,3,4$ then $\phi(I)$ would lie on the boundary of $P$. Thus without loss of generality we may assume that $t_1,t_3>0$. Then $$\phi(I)_1<\phi(J)_1, \phi(I)_2>\phi(J)_2, \phi(I)_3<\phi(J)_3, \phi(I)_4>\phi(J)_4,$$ and hence there exist $i_2,i_4 \in I\setminus J$ and $j_1,j_3 \in J \setminus I$ such that\\ $j_1<i_2<j_3<i_4$. This contradicts the assumption that $\C$ is weakly separated. We can similarly prove the statement for $-P$, so we are done.
\end{proof}
\begin{lemma}\label{lem:octahedron}
Let $\C \subset {[n] \choose k}$ be a weakly separated collection, $S \in {[n] \choose k-2}$ and $a,b,c,d \notin S$ such that the following holds:
\begin{enumerate}
  \item $a<b<c<d$ (modulo $n$).
  \item $S \cup \{a,b\}, S \cup \{b,c\}, S \cup \{c,d\}, S \cup \{a,d\}, S \cup \{a,c\} \in \C$.
\end{enumerate}
Let $\hat{\C}$ be the collection that is obtained from $\C$ by applying a square move on $S \cup \{a,c\}$, so $\hat{\C}=(\C\setminus S \cup \{a,c\}) \cup (S \cup \{b,d\})$. Then the following holds:
\begin{enumerate}
  \item If the 4 numbers $a,b,c,d$ belong to the 4 different intervals $[1,x],[x+1,x+y],[x+y+1,x+y+z],[x+y+z+1,x+y+z+w]$ then
  \[\phi(S \cup \{b,d\})=\phi(S \cup \{a,c\})\pm (\alpha_1+\alpha_3)=\phi(S \cup \{a,c\})\pm(-1,1,-1,1).\]
  \item Otherwise, $\phi(\C)=\phi(\hat{\C})$.
\end{enumerate}
\end{lemma}
\begin{proof}
The first claim is clear, so we consider the second claim. Without loss of generality, there are 4 cases that we need to check:
\begin{enumerate}
  \item $a,b,c,d$ belong to the same interval. In this case $$\phi(S \cup \{a,c\})=\phi(S \cup \{b,d\}).$$
  \item $a,b,c \in [1,x]$ and $d \notin [1,x]$. Then $$\phi(S \cup \{a,c\})=\phi(S \cup \{b,c\}) \textrm{ and } \phi(S \cup \{b,d\})=\phi(S \cup \{c,d\}).$$
  \item $a,b\in [1,x]$ and $c,d\in [x+1,x+y]$. Then $$\phi(S \cup \{a,c\})=\phi(S \cup \{b,c\}) \textrm{ and } \phi(S \cup \{b,d\})=\phi(S \cup \{a,d\}).$$
  \item $a,b\in [1,x]$, $c\in [x+1,x+y]$, $d\in [x+y+1,x+y+z]$. Then $$\phi(S \cup \{a,c\})=\phi(S \cup \{b,c\}) \textrm{ and } \phi(S \cup \{b,d\})=\phi(S \cup \{a,d\}).$$
\end{enumerate}
In all the cases, $\phi(\C)=\phi(\hat{\C})$ so we are done.
\end{proof}
We are now ready to present the proof of Theorem~\ref{thm:twointervals}. Throughout the proof, we let $\phi:=\phi_{p_1,p_2,p_3,p_4}$.
\begin{proof}
Let $\C,\hat{\C} \in {[2k] \choose k}$ be two maximal weakly separated collections that contain $A$ and $\AA$ respectively (since $A$ and $\AA$ are not weakly separated, $\C \neq \hat{\C}$). Denote by
\begin{equation}\label{eq:mutationsequence}
\C_0=\C \rightarrow \C_1 \rightarrow \C_2 \rightarrow \ldots \rightarrow \C_m = \hat{\C}
\end{equation}
the shortest sequence of square moves that transforms $\C$ into $\hat{\C}$. Note that $\phi(A)=(p_1,0,p_3,0), \phi(\AA)=(0,p_2,0,p_4)$ and $p_1+p_3=p_2+p_4=k$. Since rotating $A$ and $\AA$ (modulo $2k$) and switching between them has no effect on the result, we can assume without loss of generality that $0 < p_2,p_3,p_4 \leq p_1$ (and hence $p_3 \leq p_1,p_2,p_4$).
Let $P$ and $Q$ be the square pyramids consisting of the points $\phi(A)+\sum_{i=1}^4 t_i\alpha_i$ and $\phi(\AA)+\sum_{i=1}^4 t_i(-\alpha_i)$ respectively, such that $t_i \geq 0$ for all $i$. We use $\alpha_i$'s from Lemma~\ref{lemma:abovepyramid}.
By this lemma, all the points in $\phi(\C)$ cannot lie in the interior of $P$ (so they lie either on the surface or outside the pyramid). Similarly, all the points in $\phi(\hat{\C})$ cannot lie in the interior of $Q$. Recall that $p_1-p_4+p_3=p_2, p_4 \leq p_1$, and hence
$$\phi(\AA)=\phi(A)+(p_1-p_4)\alpha_3+p_4\alpha_4+p_3\alpha_2$$
and $\phi(\AA)$ lies in the interior of the pyramid $P$.
Therefore, the interior of $P\cap Q$ does not contain points of $\phi(\C)$ or $\phi(\hat{\C})$. Let $Z$

be the set of points that lie in the intersection of the interior of $Q$ with the set
$$\{x \mid x=\phi(A)+t_i\alpha_i+t_{i+1}\alpha_{i+1} \textrm{ for } i \in \{1,2,3,4\}, 0 \leq t_i,t_{i+1} \in \mathbb{Z}\}$$
(we define $i+1=1$ for $i=4$). Note that $Z$ consists of all the "integral" points on the boundary of $P$, that lie in the interior of $Q$ (an integral point is a point for which all the four $t_i$'s are integers).

Lemma~\ref{lem:octahedron} implies that for any $0 \leq i \leq m-1$, either $\phi(\C_i)=\phi(\C_{i+1})$, or $\phi(\C_{i+1})$ is obtained from $\phi(\C_i)$ by taking a point $O \in \phi(\C_i)$ and adding to it (or subtracting form it) $\alpha_1+\alpha_3=(-1,1,-1,1)$. By the lemma, the latter option is possible only if all the four points $O+\alpha_1, O+\alpha_2, O+\alpha_3, O+\alpha_4$ (or all the four points $O-\alpha_1, O-\alpha_2, O-\alpha_3, O-\alpha_4$) are in $\phi(\C_i)$. Recall that the image of $\phi$ lies in the 3-dimensional space $F$ that formed by the vectors in $\mathbb{R}^4$ whose sum is $k$. Let $N$ be the plane that is orthogonal to the line $t(-1,1,-1,1)$ in $F$. Finally, let $g:F \rightarrow N$ be the orthogonal projection onto $N$. By Lemma~\ref{lem:octahedron} and the discussion above,
\begin{equation}\label{eq:projection_g}
g(\phi(\C_i))=g(\phi(\C_j)) \textrm{ for all } 0 \leq i < j \leq m.
\end{equation}
In addition, since $\alpha_1+\alpha_3=\alpha_2+\alpha_4=(-1,1,-1,1)$, then for a point $O$ as above we have
$$g(\{O+\alpha_1, O+\alpha_2, O+\alpha_3, O+\alpha_4\})=g(\{O-\alpha_1, O-\alpha_2, O-\alpha_3, O-\alpha_4\}).$$
This holds since $O-\alpha_1+(-1,1,-1,1)=O+\alpha_3, O-\alpha_2+(-1,1,-1,1)=O+\alpha_4$, and $N$ is orthogonal to $(-1,1,-1,1)$.
Consider now the sequence in (\ref{eq:mutationsequence}). Clearly, $g(\phi(A))\in g(\phi(\C)) \cap g(Z)$. The set $A$ must be mutated at some stage in the sequence, since $ A \notin \C_m$. Therefore, from the discussion above, each one of $g(\phi(A)+\alpha_1)$, $g(\phi(A)+\alpha_2)$, $g(\phi(A)+\alpha_3)$ and $g(\phi(A)+\alpha_4)$ must lie in $g(\phi(\C_j))$ for some $j$. By (\ref{eq:projection_g}) we get that each one of them lies in $g(\phi(\C))$, and therefore corresponds to an element $E_i$ in $\C$ (for $i=1,2,3,4$ respectively). For each such $i$, if $g(\phi(E_i)) \in g(Z)$ then $E_i$ lies above (or on the surface) of $P$, and does not belong to $\C_m$. Therefore it should be mutated at some stage in (\ref{eq:mutationsequence}). Hence applying to $E_i$ the same argument that we applied to $A$ leads to the fact that each one of $g(\phi(E_i)+\alpha_1)$, $g(\phi(E_i)+\alpha_2)$, $g(\phi(E_i)+\alpha_3)$ and $g(\phi(E_i)+\alpha_4)$ lies in $g(\phi(\C))$. Continuing this way we get that $g(Z) \subset g(\phi(\C))$. Let $g(Z)=\{x_1,\ldots,x_l\}$. From the definition of $Z$, $g$ acts injectively when restricted to $Z$, and hence
$l=|Z|$.
Therefore, there are at least $l$ elements that lie in $\C$ and do not lie in $\hat{\C}$. We will now show that
\begin{equation}\label{eq:l}
l=1+k^2-2k-\sum_{i=1}^{4} {p_i\choose 2}.
\end{equation}
Note that this would imply that the number of elements in $\C \cap \hat{\C}$ is at most \[k^2+1-\Big(1+k^2-2k-\sum_{i=1}^{4} {p_i\choose 2}\Big)=2k+\sum_{i=1}^{4} {p_i\choose 2},\]
and since this holds for any pair of maximal weakly separated collections $\C,\hat{\C}$ for which $A \in \C,\AA \in \hat{\C}$, this implies the first part of (1) in Theorem~\ref{thm:twointervals}. In order to show that this bound is tight, we can use the same construction as in the proof of Theorem~\ref{thm:main}, and that will conclude the proof of (1). Let us now prove equation (\ref{eq:l}). As the cases $p_2 \leq p_4$ and $p_2 \geq p_4 $ can be handled similarly, we will assume without loss of generality that $p_2 \geq p_4$. First, observe that the structure of $g(Z)$ is the interior portion of the shape in Figure~\ref{fig:interiorP1}. For specific examples see Figures~\ref{fig:P4323andP6314} and ~\ref{fig:rotations}. We will count the number of integer points in the interior of this shape. The interior of the square with $p_3$ side length contains $x_1=4{p_3+1 \choose 2}-4p_3+1$ integer points. The remainder parts of the right and the left sections contain $x_2=p_3(p_4-p_3)$ and $x_3=p_3(p_2-p_3)$ integer points respectively. The remainder part of the bottom section contains $x_4=2{p_2 \choose 2}-{p_3 \choose 2}-{p_3-1 \choose 2}+(p_2-1)(p_4-p_2-1)$ points. Note that for the calculation of $x_4$, we added two "large" triangles (with $p_2$ base), subtracted two "small" triangles (with $p_3$ base), and added the parallelogram with sides $p_2$ and $p_4-p_2$. Thus,
$l=x_1+x_2+x_3+x_4$. We need to show that this sum equals the RHS of (\ref{eq:l}). First, rewrite the RHS of (\ref{eq:l}) as $4\Big(\frac{k-1}{2}\Big)^2-\sum_{i=1}^{4}\frac{p_i(p_i-1)}{2}$. Thus
\begin{equation}
\begin{aligned}
RHS ={ }{ } & \Big(\frac{p_1}{2}+\frac{p_3-1}{2}\Big)^2+\Big(\frac{p_3}{2}+\frac{p_1-1}{2}\Big)^2+
\Big(\frac{p_2}{2}+\frac{p_4-1}{2}\Big)^2+ \\
      & +\Big(\frac{p_4}{2}+\frac{p_2-1}{2}\Big)^2-2\sum_{i=1}^{4}\frac{p_i(p_i-1)}{4}.\\
\end{aligned}
\end{equation}
Expanding the expression above, we get
\begin{equation}
\begin{aligned}
RHS ={ }{ } & \sum_{i=1}^{4} \Big(\frac{p_i}{2}\Big)^2+\Big(\frac{p_i-1}{2}\Big)^2-\frac{2p_i(p_i-1)}{4}+\\
      & + \frac{p_1(p_3-1)}{2}+\frac{p_3(p_1-1)}{2}+\frac{p_2(p_4-1)}{2}+\frac{p_4(p_2-1)}{2},\\
\end{aligned}
\end{equation}
which equals
$$\sum_{i=1}^{4} \Big(\frac{p_i}{2}-\frac{p_i-1}{2}\Big)^2+p_1p_3+p_2p_4-\frac{p_1+p_2+p_3+p_4}{2}= 1+p_1p_3+p_2p_4-k.$$

Therefore, we need to show that $x_1+x_2+x_3+x_4=1+p_1p_3+p_2p_4-k$. Observe that
\[x_1+x_2+x_3+x_4-(1+p_1p_3+p_2p_4-p_2-p_4)=p_3(p_2+p_4-p_1-p_3)=p_3(k-k)=0,\]
and we are done with (1). Part (2) follows directly from part (1). For part (3), from the proof of part (1) we get that $D(A,\AA)$ is bounded from below by the number of integral points in the interior of $P \cap Q$. This looks like a cuboid with sides $p_2,p_3,p_4$ with pyramids removed from opposite sides, resulting in $p_2p_3p_4-2{p_3+1 \choose 3}$ integral points. In order to show that this bound is tight, one needs to construct a pair of weakly separated sets $C,\hat{C}$ that satisfy the following properties: $A \in C, \AA \in \hat{C}$, and $\phi(C)$ ($\phi(\hat{C})$) contains the surface of $P$ ($Q$) that lies in the interior of $Q$ ($P$). This can be done easily using plabic graphs that have a form of a honeycomb (see section 12 in \cite{FP}), so we are done.

\begin{figure}[h]
\centering
\begin{tikzpicture}[scale=0.4]
\draw (5,0) -- (9,0) -- (12,3) -- (12,7) -- (7,12) -- (3,12) -- (0,9) -- (0,5) --cycle;
\draw (3,12) -- (12,3) -- (2,3);
\draw (7,12) -- (0,5) -- (10,5);
\draw (3,12) -- (3,8) -- (7,8) -- (7,12);
\draw[red,thick] (3,12) -- (7,12) node[midway,above,black] {$p_3$};
\draw[red,thick] (7,8) -- (7,12) node[midway,right,black] {$p_3$};
\draw[red,thick] (0,5) -- (5,10) node[midway,left,black] {$p_2$};
\draw[red,thick] (0,5) -- (10,5) node[midway,above,black] {$p_2$};
\draw[red,thick] (5,10) -- (12,3) node[midway,right,black] {$p_4$};
\draw[red,thick] (12,3) -- (9,0) node[midway,right,black] {$p_2-p_3$};

\end{tikzpicture}
\caption{The length of the red lines is denoted by the symbol next to them. By \emph{length} we mean the number of integer points in the red segment+1.}
\label{fig:interiorP1}
\end{figure}
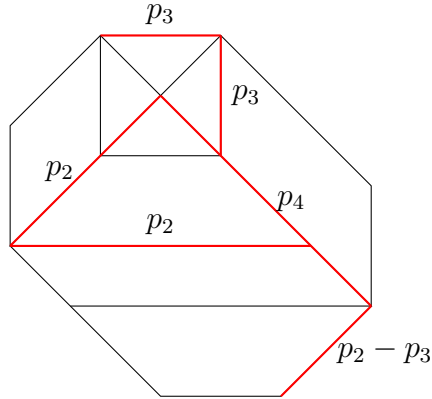

\begin{figure}[h]
\centering
\includegraphics[height=2in]{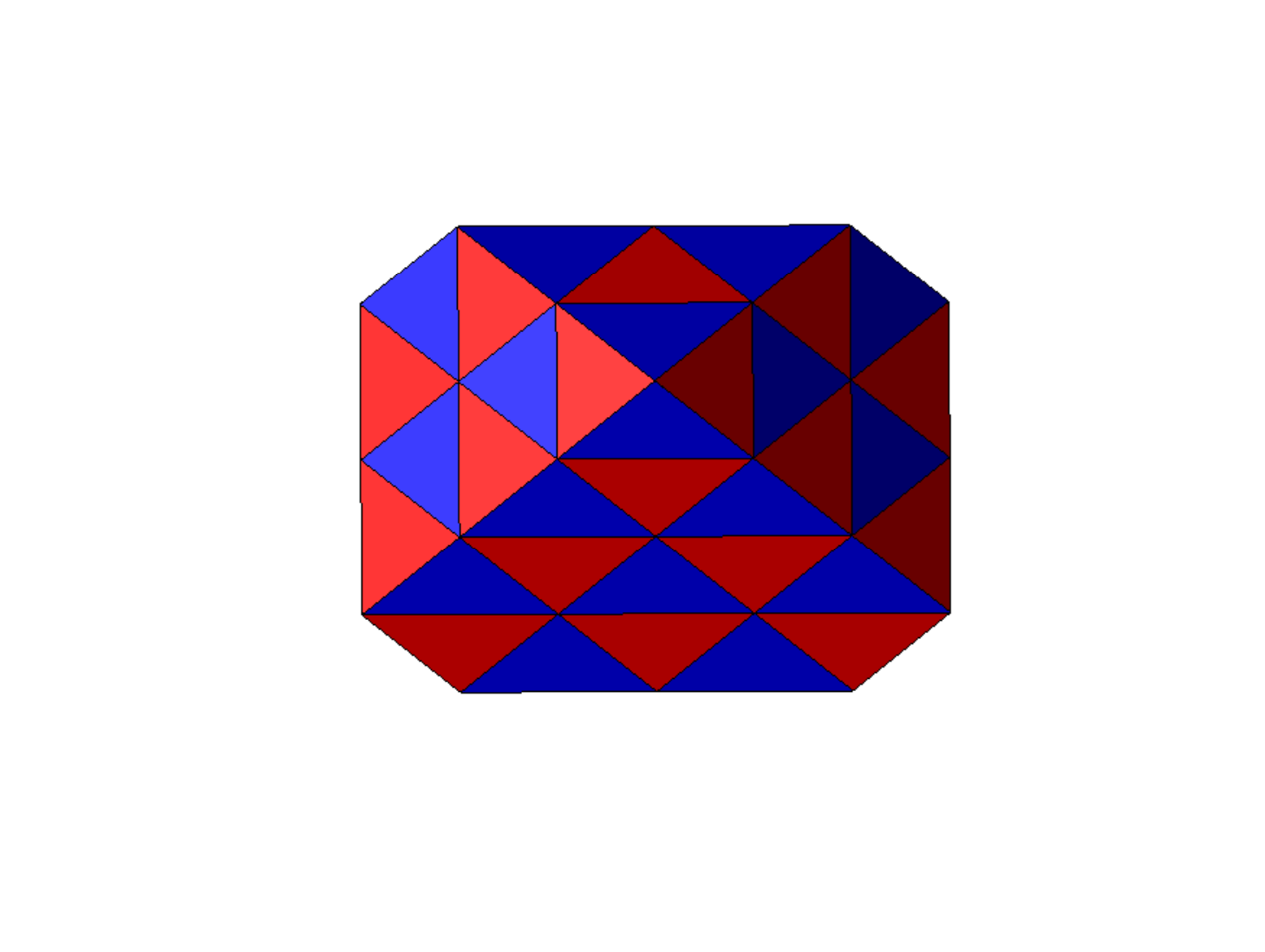}%
\includegraphics[height=2in]{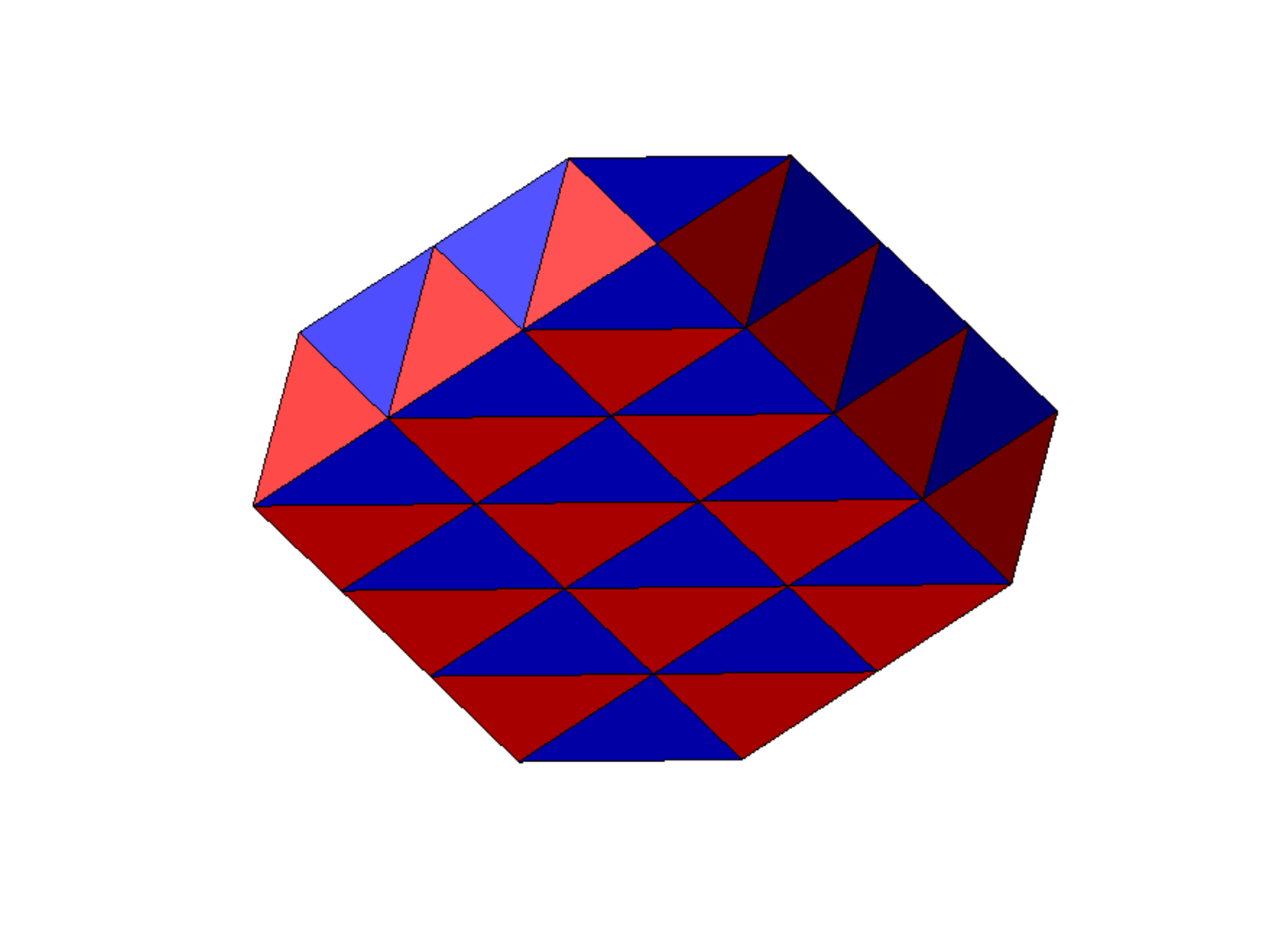}%
\caption{The left figure corresponds to the case $p_1=4,p_2=3,p_3=2,p_4=3$. The right Figure corresponds to the case $p_1=6,p_2=3,p_3=1,p_4=4$. Note that these figures match the description in Figure~\ref{fig:interiorP1}. For example, the top horizontal line has $p_3+1$ integer points, as indicated in the description of Figure~\ref{fig:interiorP1}.}
\label{fig:P4323andP6314}
\end{figure}

\begin{figure}[h]
\centering
\includegraphics[height=1.1in]{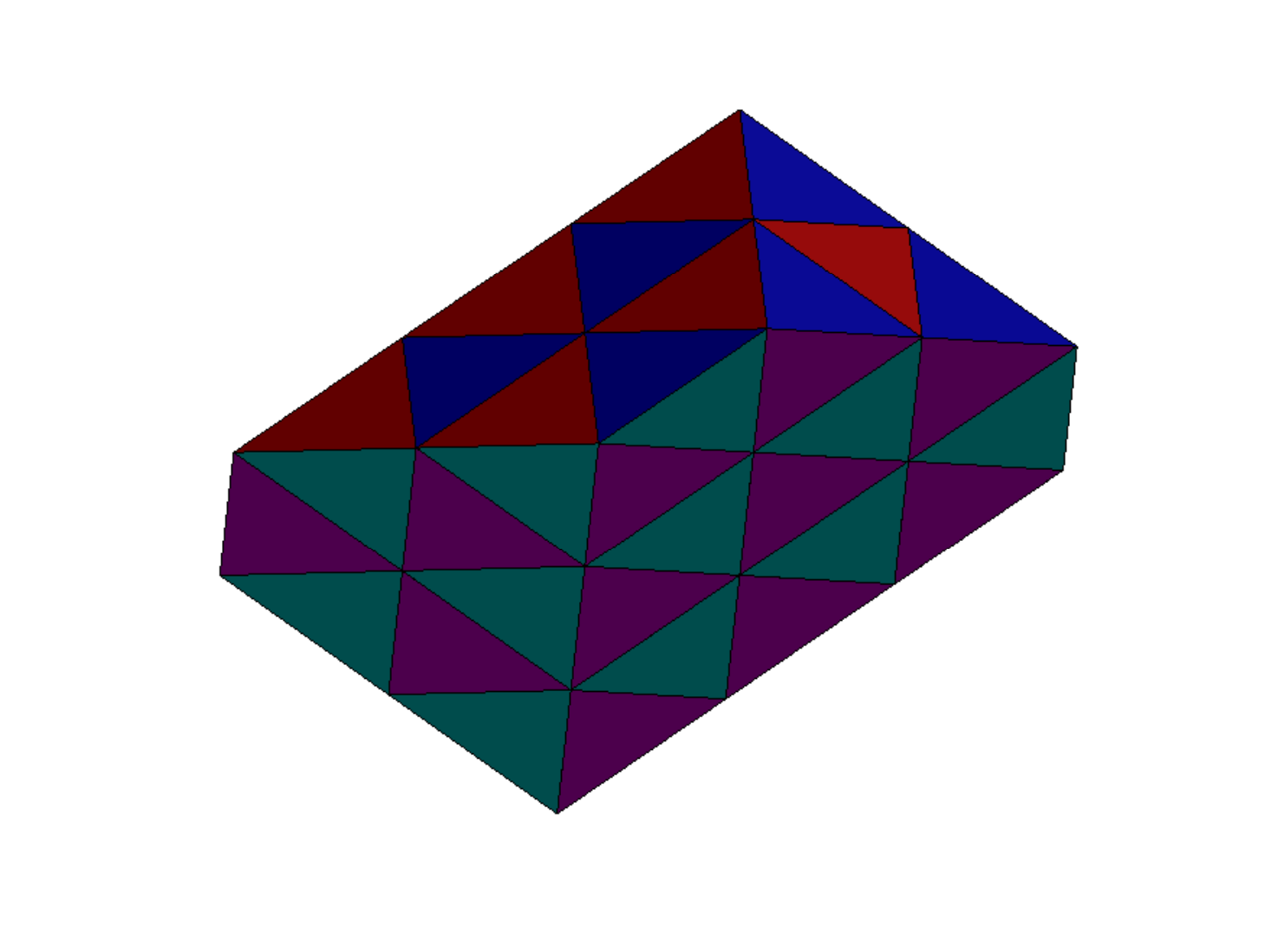}%
\includegraphics[height=1.1in]{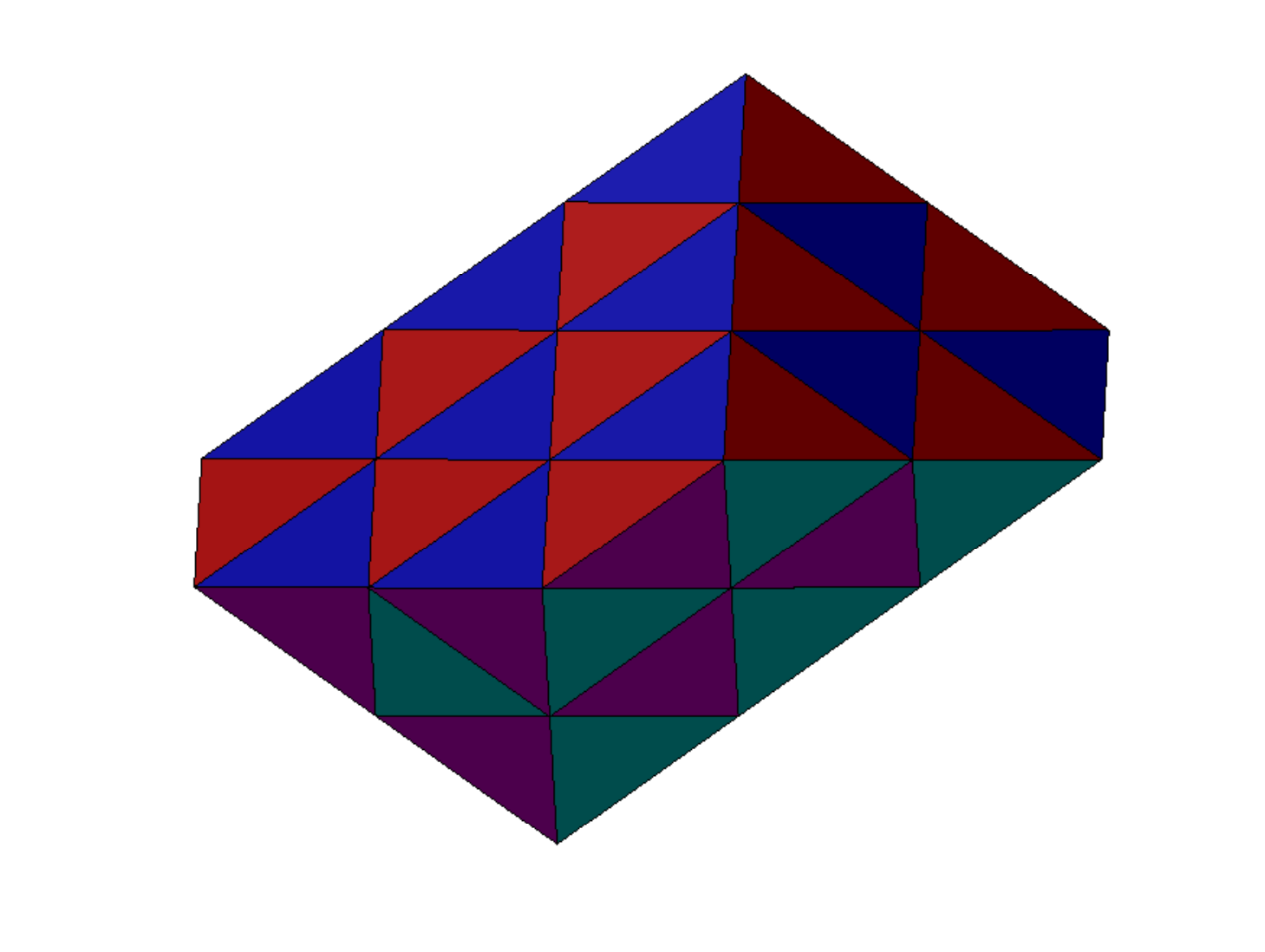}%
\includegraphics[height=2in]{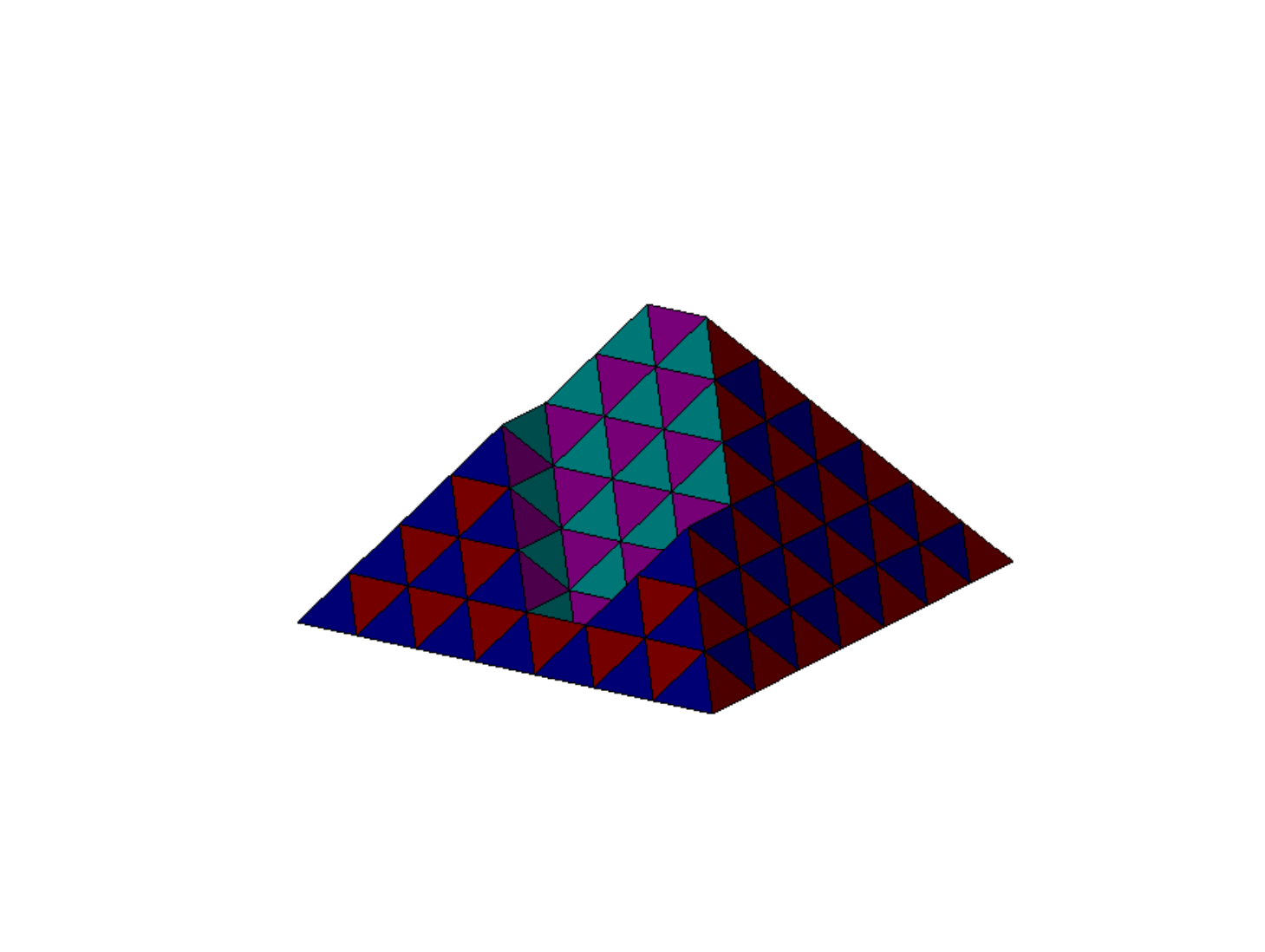}%
\caption{The left two figures are rotations of the shape formed by the intersection of the pyramids $P$ and $Q$ for the case $p_1=4,p_2=3,p_3=2,p_4=3$. The right figure shows the remainder of the pyramid $P$ after we removed the portion which intersects with $Q$ for the case $p_1=6,p_2=3,p_3=1,p_4=4$.}
\label{fig:rotations}
\end{figure}

\end{proof}
We conclude this section with the proof of Theorem~\ref{thm:lowerboununbalanced}.
\begin{proof}[Proof of Theorem~\ref{thm:lowerboununbalanced}]
In order to prove this theorem, it is enough to construct a weakly separated collection in $\A_{A,\AA}$ of cardinality (\ref{eq:unbalanced}). Let $S_i=[s,\ldots,s+p_i-1]$ be an interval in the partition of the circle associated with $A$, and let $x<y \in S_i$. Define $T_i$ to be the collection of all the $k$-tuples of the form $[x-k+s+p_i-y,x-1] \cup [y,s+p_i-1]$. Note that $|T_i|={p_i \choose 2}$. Now, for any $i<j$ such that $i \neq j-1$, let $[s,s+a_i-1], [t,t+b_j-1]$ be the two corresponding intervals in the partition of the circle associated with $A$. Define $M_{ij}$ to be the collection of all $k$-tuples of the form $[x,s+a_i-1] \cup [y,t+b_j-1]$ such that $x \in [s,s+a_i-1], y \in [t,t+b_j-1]$. Note that there are $\chi_{i,j}(a_i+b_j-k+1)$ such $k$-tuples. Finally, for all $i$ denote by $B_i$ the interval $[i,i+k-1]$. In order to prove the theorem, it is enough to show that
\[\C=\left(\cup_{i}T_i\right) \bigcup \left(\cup_{ij}M_{ij}\right) \bigcup \left(\cup_{i}B_i\right)\]
is a weakly separated collection in $\A_{A,\AA}$ (note that from the definition of $T_i, M_{ij}$ and $B_i$, all the $k$-tuples in the union are different). First, it is easy to see that every $k$-tuple in $\C$ is weakly separated from both $A$ and $\AA$. This holds since for any $I \in \C$, $I \setminus A$ is either an interval, or a union of two intervals such that no element of $A \setminus I$ lies in between them (from one of the two sides). The same holds for $\AA$. Hence $\C \subset \A_{A,\AA}$. We now need to show that $\C$ is weakly separated. Now, if $M_{ij}$ is nonempty for some pair $i<j, i \neq j-1$, then $a_i+b_j \geq k$. This implies that for all other nonempty sets of the form $M_{uv}$, either we always have $u=i$, or we always have $v=j$. From here, by simple case analysis, it is easy to verify that $\C$ is indeed weakly separated, and we are done.
\end{proof}

\subsection*{Acknowledgments}
We thank Alex Postnikov for introducing us into the subject and a lot of fruitful conversations.

\end{document}